\documentclass[11pt,reqno,a4paper]{amsart}

\usepackage{stmaryrd}
 \usepackage{enumitem}
 \usepackage{amsmath,amssymb,amsfonts,amsthm,bm}
 \usepackage[numbers, sort&compress]{natbib}
 \usepackage[left=1.35in, right=1.35in]{geometry}
 
 \newtheorem{theorem}{Theorem}[section]
 
 \newtheorem{lemma}[theorem]{Lemma}
 \newtheorem{proposition}[theorem]{Proposition}
 \theoremstyle{definition}
 \newtheorem{definition}[theorem]{Definition}
 \theoremstyle{remark}
 \newtheorem{remark}[theorem]{Remark}
 
 \numberwithin{equation}{section}

\newcommand{\eps}{\varepsilon}

\newcommand{\abs}[1]{\left\vert#1\right\vert}
\newcommand{\set}[1]{\left\{#1\right\}}

\newcommand{\norm}[1]{\Vert#1\Vert}
\newcommand{\normm}[1]{| \! | \! | #1| \! | \! | }
\newcommand{\inner}[1]{\left(#1\right)}
\newcommand{\comi}[1]{\left<#1\right>}
\newcommand{\comii}[1]{\left<#1\right>}
\newcommand{\com}[1]{\left[#1\right]}
 
\begin{document}

%
%
%
%
%
%
%
%
%

\title[Gevrey smoothing effect for the Boltzmann equation]
 {Gevrey smoothing effect for   the spatially inhomogeneous  Boltzmann equations without  cut-off}

\author[H.Chen]{Hua Chen}

\address[H.Chen]{School of Mathematics and Statistics, and Computational Science Hubei Key Laboratory,   Wuhan University,  430072 Wuhan, China}

\email{chenhua@whu.edu.cn}


\author[X. Hu]{Xin Hu}
\address[X. Hu]{School of Mathematics and Statistics,   Wuhan University,  430072 Wuhan, China}
\email{hux@whu.edu.cn}

\author[W.-X. Li]{Wei-Xi Li}
\address[W.-X. Li]{School of Mathematics and Statistics, and Computational Science Hubei Key Laboratory,   Wuhan University,  430072 Wuhan, China}
\email{wei-xi.li@whu.edu.cn}

\author[J. Zhan]{Jinpeng Zhan}
\address[J. Zhan]{School of Mathematics and Statistics,   Wuhan University,  430072 Wuhan, China}
\email{jinpeng@whu.edu.cn}



\thanks{The work  was supported by Fok Ying Tung Education Foundation (151001) and NSFC (11771342, 11422106)}
\date{}

\begin{abstract}
In this article we study the Gevrey regularization effect  for  the spatially inhomogeneous  Boltzmann equation without angular cutoff.   This equation is partially elliptic in the velocity direction and degenerates in the spatial variable.     We consider the nonlinear Cauchy problem for the fluctuation around the Maxwellian distribution and prove  that any solution with mild regularity will become  smooth in Gevrey class  at positive time,  with Gevrey index  depending on the angular singularity.   Our proof relies on the symbolic calculus for the collision operator and the global subelliptic estimate  for the Cauchy problem of  linearized Boltzmann operator. 
\end{abstract}

\maketitle
   \section{Introduction and main result}

The Cauchy problem for the spatially inhomogeneous  Boltzmann equation reads
      \begin{equation}\label{1}
       \begin{array}{ll}
       \partial_{t}F+v\cdot\partial_{x}F=Q(F,F), \quad F|_{t=0}=F_0,
       \end{array}
      \end{equation}
       where $F(t,x,v)$ is a probability density   with a  given  datum $F_0$ at $t=0,$
       and   $x$  and $v$ stand respectively for the spatial and velocity variables and we   consider here the important   physical dimension $n=3$ and suppose   both vary in the whole space $\mathbb{R}^3.$   When the density function $F$ doesn't depend on the spatial variable $x$ we get the spatially homogeneous Boltzmann equation:
       \begin{eqnarray*}
       	 \partial_{t}F =Q(F,F), \quad F|_{t=0}=F_0. 
       \end{eqnarray*} 
       The  bilinear operator  $Q$ on the right-hand side of \eqref{1}  stands for the collision part  acting  only on the velocity,  so    the spatially inhomogeneous  Boltzmann equation degenerates in $x,$ which is one of the main  difficulties in the regularity theory.  In addition to the degeneracy, another major difficulty  arises from the  nonlocal property of the collision operator $Q,$ which is defined      
      for suitable functions $F$ and $G$ by
      \begin{equation}\label{collis}
       Q(G,F)(t,x,v)=\int_{\mathbb{R}^3}\int_{\mathbb S^2}B(v-v_*,\sigma)(G^{\prime}_{*}F^{\prime}- G_* F)dv_{*}d\sigma,
       \end{equation}
where  and throughout the paper we write $F^{\prime}=F(t,x,v^{\prime}), ~F=F(t,x,v), ~G^{\prime}_{*}=G(t,x,v^{\prime}_{*})$ and $G_*=G(t,x,v_*)$ for short, and the pairs  $(v, v_*)$ and   $(v^{\prime}, v^{\prime}_*) $   stand  respectively for  the velocities of  particles before and after collision, with the following momentum and  energy conservation rules fulfilled,
\begin{equation*}
      v^{\prime}+v^{\prime}_{*}=v+v_{*} , ~~|v^{\prime}|^{2}+|v^{\prime}_{*}|^{2}=|v|^{2}+|v_{*}|^{2}.
\end{equation*}
From the above relations   we have  the so-called $\sigma$-representation,  with $\sigma\in\mathbb S^{2},$
 \begin{equation*}
 \left\{
 \begin{aligned}
   &v'  = \frac{v+v_*}{2} + \frac{ |v-v_*|}{2} \sigma, \\
  &v'_* = \frac{v+v_*}{2} - \frac{ |v-v_*|}{2} \sigma.
 \end{aligned}
 \right.
\end{equation*}
The cross-section $B(v-v_*,\sigma)$ in \eqref{collis}  depends on the relative velocity $\abs{v-v_*}$ and the  deviation angle $\theta$  with 
 \begin{equation*}
 \cos \theta = \big  \langle \inner{ v-v_*}/|v-v_*|, \ \sigma\big\rangle.
\end{equation*}
Here  we denote by $\langle \cdot, \cdot\rangle$  the scaler product in $\mathbb R^3.$   
Without loss of generality we may assume that $B(v-v_*,\sigma)$ is supported
on the set $\theta\in [0,\pi/2]$ where $\langle v-v_*,\sigma\rangle \geq 0$,  since as usual
 $B$ can be replaced  by its symmetrized version, and  furthermore  we  may suppose it takes the following form:
  \begin{equation}\label{kern}
 { B}(v-v_*, \sigma) =  |v-v_*|^\gamma  b (\cos \theta),
 \end{equation}
 where $\gamma\in ]-3, 1].$ Recall $\gamma=0$ is the Maxwellian molecules case and meanwhile  the cases of  $-3<\gamma< 0$ and  $0<\gamma\leq1$  are called respectively  soft   potential and hard potential.     In this paper we will restrict our attention to  the cases of  Maxwellian molecules and hard potential, i.e., $0\leq \gamma\leq1.$   
 Furthermore we are concerned about   singular cross-sections, also called non cut-off sections,  that is,  the angular part $b(\cos\theta)$ has singularity near $0$ so that 
 \begin{eqnarray*}
	\int_0^{\pi/2}\sin \theta   b(\cos \theta)\ d\theta=+\infty.
\end{eqnarray*}
 Precisely, we suppose  $b$ has the following expression  near $\theta=0$:  
  \begin{equation}\label{angu}
 0 \leq  \sin \theta   b(\cos \theta)  \approx   \theta^{-1-2s},
 \end{equation}
 where and throughout the paper  $p\approx q$ means  $C^{-1}q\leq p\leq Cq$ for
some  constant $C\geq 1$.   Note that the cross-sections of  type \eqref{kern}  include the potential of inverse power law as a  typical physical model.

It is well understood nowadays  that wether or not the angular singularity occurs      is closely linked with the regularization effect and propagation property in time.    If the angular collision kernel is integrable (also called Grad's cut-off assumption),  then  similar as  hyperbolic  equations   the singularity or regularity of solutions to Boltzmann equation  usually  propagates in time,   that is  the solutions should have precisely  the same singularity or regularity as initial data.  The understanding of this propagation property has made very substantial development,  and we just  mention the resutls of propagation in the Gevrey class setting  by      Desvillettes-Furioli-Terraneo \cite{MR2465814} and  Ukai \cite{Ukai84},    the argument therein  working  well for both cut-off and non cut-off cases.    The  subject of  cut-off Boltzmann equation  has a long history and  there is a vast literature on it, investigating the well-posedness, propagation property,  moments and positivity and so on.   For the mathematical treatment of cut-off Boltzmann equation, we  refer to   the books of  Cercignani  \cite{MR1069558} and Cercignani-Illner-Pulvirenti \cite{MR1307620} for instance,   and  more  classical references,  concerned with the cut-off and non cut-off cases , can be found in  the  surveys of Alexandre  \cite{MR2556715} and  Villani \cite{Villani02}.  

When the singularity is involved,  
     the properties herein are quite different from the ones observed in the cut-off case. In fact,  regularization effect occurs for the Cauchy problem of non cut-off Boltzmann equation,  due to    diffusion properties   caused by the angular singularity. 
     Then the solution should become smooth at positive times, as it does for solutions to the heat equation.  The mathematical treatment of the regularization properties  goes back to Desvillettes \cite{MR1324404,Desv97} for a one-dimensional model of the Boltzmann equation.             Later on,   Alexandre-Desvillettes-Villani-Wennberg  \cite{ADVW00}  establish the optimal regular estimate in $v$ for the collision operator after the earlier work of Lions \cite{ Lions98},  and since then substantial developments have been achieved, cf. \cite{
AHL, MR2679369, MR2847536, AMUXY1, MR2807092,MR2784329, Mou2, Mou1}   for instance and the references therein.     These works    show   that Boltzmann operator behaves locally as a fractional Laplacian:
     \begin{eqnarray*}
        (-\triangle_v)^s+\textrm{lower order terms},
     \end{eqnarray*}
     and  more precisely,  from the point of the   global  view the linearized Boltzmann operator around Maxwellian distribution behaves essentially as  
\begin{eqnarray*}
  \comii
    v^\gamma\inner{-\triangle_v-\inner{v\wedge \partial_v}^2+\abs v^2}^s+\textrm{lower order terms},
\end{eqnarray*}
  where $v\wedge \partial_v$ is the cross product of vectors $v$ and $\partial_v.$
 This diffusion property indicates that the spatially homogeneous equation should   behave as   fractional  heat equation,  and we may expect    solutions to Cauchy problem will enjoy better regularity  at positive time than    initial's.   Strongly related to this regularization effect  is another well known  Landau equation, taking into account all grazing collisions.    So far there have been
extensive works  on 
the  regularity,   in a wide variety of different settings,  of solutions to the homogeneous Boltzmann equation
without angular cut-off;  see for instance  \cite{ADVW00, MR2149928, MR2476677,  MR2557895, MR2425602,
  MR1324404,Desv97, DesvVillani00-2, DesvWennberg04, MR3680949, MR2425608, MR3485915,
  MR2523694,  MR2476686, MR2679746,  MR3325244, Villani1999} and references therein.    We refer to the very recent work of Barbaroux-Hundertmark-Ried-Vugalter \cite{Barbaroux2017},  where they  prove  any weak solution of the fully non-linear homogenous Boltzmann equation for Maxwellian molecules belongs to the Gevrey class   at   positive time, and the Gevrey index therein is optimal.

   Compared with the homogeneous case,  the situation  becomes more intricate
  for spatially inhomogeneous Boltzmann equation,  and  much less is known for the regularization properties. The main difficulty  lies in the degeneracy in spatially variable since diffusion only occurs in the velocity, and this is quite different from  the spatially homogeneous case where  we have  elliptic properties for solutions to Boltzmann equation.      Nevertheless  we may  expect  some hypoelliptic effects    due to  the non trivial interaction between the transport operator  and  the
collision operator.   To see this let us first mention the  velocity-averaging lemma, which  is  an important  tool for  transport equations and  is also applied extensively to       the study   of the Boltzmann equation.  The velocity averaging Lemma shows  the 
 velocity-averages   of solutions to transport equations    are smoother in spatial variable than the distribution function  itself;  see for instance the works of Golse-Lions-Perthame-Sentis \cite{MR923047}  and   Golse-Perthame-Sentis \cite{MR808622}.   Other tools from microlocal analysis are also developed for the hypoelliptic properties of  Boltzmann equation in the setting of   $L^2$ norm,   and  we refer the interested readers  to the works of Bouchut \cite{Bouchut02}  for  the use of  H\"ormander's  techniques  and  Alexandre-Morimoto-Ukai-Xu-Yang \cite{ALEXANDRE20082013} for the application of  uncertainty principle to  kinetic equations,  as well as the work \cite{AHL} involving  the multiplier method.   To understand the intrinsic hypoelliptic structure     
Morimoto-Xu \cite{MR2359105} initiated to study the following  simplified Boltzmann model 
\begin{eqnarray*}
\partial_t+v\partial_x+\sigma_0\big(-  \Delta_v\big)^s,
\end{eqnarray*}  
 where $\sigma_0>0$ is a constant, and   using the analysis of the commutator between transport part and diffusion part they obtain the subelliptic estimate in time-space variables, see also \cite{MR2763329, MR2876831}   for the further improvement on the exponent of subelliptic estimate.  Note the above operator  is just a local model of Boltzmann equation, inspirited by  the diffusion property in $v$ velocity obtained in Alexandre-Desvillettes-Villani-Wennberg  \cite{ADVW00}. 
Furthermore  in the joint work \cite{AHL} of the third author with Alexandre and H\'erau,     the global sharp estimate is obtained for the linearized Boltzmann operator rather than the model operators,  using additionally  symbolic calculus for the collisional cross-section; see also \cite{MR3102561, HK2011,MR3193940} for  the earlier  works  on  the hypoelliptic properties of other related models.    Let us mention that the aforementioned works about hypoellipticity don't involve the initial data,  and in fact  the time variable $t$ therein is supposed to vary in the whole space so that  Fourier analysis can be applied when deriving the subelliptic estimate in time variable. In this work we are concerned with the hypoelliptic structure for the Cauchy problem of Boltzmann equation and thus  the   initial data will be  involved in the analysis.   
 
Now we mention the regularity results for the spatially inhomogeneous Boltzmann equation. In fact the well-posedness for general initial data    is  still a mathematically challenging problem, and so far there are much fewer results.  In 1989,   DiPerna-Lions \cite{DipernaLions89} established global renormalized weak solutions  in the cut-off case for general initial data without a size restriction, and   Alexandre-Villani \cite{MR1857879}  in  2002  proved  the existence of DiPerna-Lions' renormalized weak solutions in non cut-off case.  Under a mild regularity assumption on the initial data, the local-in-time  existence and uniqueness are obtained by Alexandre-Morimoto-Ukai-Xu-Yang \cite{MR2679369}.   We also mention the earlier work of  Ukai \cite{Ukai84}    where the well-posedness in anisotropic Gevrey space  is established by virtue of the Cauchy-Kovalevskaya  theorem.   When considering the perturbation around Maxwellian distribution, the well-posedness in weighted Sobolev space is obtained independently by  Gressman-Strain \cite{MR2784329} and   
Alexandre-Morimoto-Ukai-Xu-Yang \cite{AMUXY1, MR2793203},  where the   novelty is the introduction of a non isotropic triple norm which  enables to capture the sharp estimate to close the energy.   We will explain later the non isotropic triple norm in detail. 
The  DiPerna-Lions' renormalized solutions are quite weak so the uniqueness is unknown.  It is natural to  expect a higher-order regularity  of weak solutions and this still remains a challenging  problem up to now.    We refer to the  very recent works of   Golse-Imbert-Mouhot-Vasseur \cite{GIMV},  Imbert-Mouhot  \cite{IM2, IM} and  Imbert-Silvestre \cite{IS} for the  progress on this regularity issue,  where  H\"older continuity of   $L^\infty$ weak solutions is obtained   by using  the Harnack inequality and De Giorgi-Nash-Moser theorem.   As for the spatially inhomogeneous Landau equation  the $C^\infty$ smoothing  of bounded weak solutions is obtained by   Henderson-Snelson \cite{2017arXiv170705710H} and Snelson \cite{2018arXiv180510264},      where  the pointwise Gaussian upper bound plays a crucial role; see also the  recent work of Imbert-Mouhot-Silvestre \cite{2018arXiv180406135I} for an attempt to establish  upper bounds   for  Boltzmann equation.        On the other hand
 a long lasting conjecture on the smoothing effect expects   better   regularity of  solutions at positive time than initial's and asks furthermore how much better.    Under a mild regularity assumption on the initial data, the $C^\infty$ smoothing effect  is obtained by Chen-Desvillettes-He \cite{ChenDesvHe07} for inhomogeneous Landau equation and  by  Alexandre-Morimoto-Ukai-Xu-Yang \cite{MR2679369, ALEXANDRE20082013}  for Boltzmann equation. In this paper we are concerned with a higher-order regularity of mild solutions at positive time, inspirited by the Gevrey regularization effect for  the  fractional heat equation,  and our main tool here will be the  symbolic calculus developed in \cite{AHL}.  Let us mention the Gelfand-Shilov and Gevrey smoothing effect has been obtained by  Lerner-Morimoto-Pravda-Starov-Xu \cite{LERNER2015459} for  non-cutoff Kac equation, a one-dimensional Boltzmann model. Here we will further to investigate the most physical  three-dimensional Boltzmann equation.    We hope the present work may give  better insights into the regularity issue of inhomogeneous  Boltzmann equation. 
 
 We will restrict our attention to  the fluctuation around the Maxwellian distribution. 
 Let 
 \begin{equation*}
      \mu(v)=(2\pi)^{-3/2}e^{-|v|^{2}/2}
\end{equation*}
be the normalized Maxwellian distribution. 
Write   solution  $F$ of \eqref{1} as  $F=\mu+\sqrt{\mu}f$ and  accordingly  $F_0=\mu+\sqrt{\mu}f_0$ for the initial datum.  Then the fluctuation $f$ satisfies the Cauchy problem
\begin{eqnarray}\label{eqforper}
\left\{
\begin{aligned}
     & \partial_{t}f+v\cdot\partial_{x}f-\mu^{-{1\over 2}}Q(\mu,\sqrt{\mu}f)-\mu^{-{1\over 2}}Q(\sqrt{\mu}f,\mu)
      =\mu^{-{1\over 2}}Q(\sqrt{\mu}f,\sqrt{\mu}f),\\
& f|_{t=0}=f_0.
      \end{aligned}
      \right.
\end{eqnarray}
We will use throughout the paper the notations as follows.  Define by $\mathcal L$ the linearized collision operator, that is 
\begin{equation}\label{linearbol}
 \mathcal L f=\mu^{-1/2}Q(\mu,\sqrt{\mu} f)+\mu^{-1/2}Q(\sqrt{\mu}f,\mu), 
 \end{equation}
 and denote 
 \begin{eqnarray*}
 	\Gamma(g,h)=\mu^{-1/2}Q(\sqrt{\mu}g,\sqrt{\mu}h).
 \end{eqnarray*}
Furthermore denote by $P$ 
  the linearized Boltzmann operator:   
\begin{equation}\label{linbol}
P=\partial_{t}+v\cdot\partial_{x}-\mathcal L. 
\end{equation}
So the Cauchy problem \eqref{eqforper} for  the perturbation $f$ can be rewritten as
\begin{equation}\label{3}
      Pf=\Gamma(f,f),\quad f|_{t=0}=f_0.
\end{equation}
Note that the global existence in Sobolev space for the above Cauchy problem is obtained by Alexandre-Morimoto-Ukai-Xu-Yang \cite{AMUXY1, MR2793203},     taking advantage of a triple norm  defined by 
   \begin{equation}\label{trnor}
   \normm f^2\stackrel{\rm def}{=} \int B(v-v_*,\sigma) \mu_*
\inner{f-f'}^2 +\int B(v-v_*,\sigma) f_*^2 \inner{\sqrt{\mu'}-
\sqrt{\mu}}^2,
 \end{equation}
  where the integration is over $  \mathbb R_v^3\times\mathbb R_{v_*}^3\times\mathbb
S_\sigma^2.$  And we refer  to the work of Gressman-Strain \cite{MR2784329} for the global  existence in Sobolev space when $x$ varies in a torus. Denote by $H^k(\mathbb R^6)$ the classical Sobolev space. For any $\ell\in \mathbb R$ define
\begin{eqnarray*}
	H^k_\ell(\mathbb R^6)=\set{u\in H^k(\mathbb R^6); \quad \comi v^\ell\in H^k(\mathbb R^6)},
\end{eqnarray*}  
where and throughout the paper we use the notation
\begin{eqnarray*}
	\comi{\cdot}=\inner{1+\abs{\cdot}^2}^{1/2}.
\end{eqnarray*}

\begin{theorem}
[Global existence in  \cite{MR2793203}]
\label{thex}	 Assume that the cross-section satisfies \eqref{kern} and \eqref{angu}  with $0 < s < 1$  and $\gamma + 2 s > 0$.  Suppose 
the initial data  $f_0\in H^k_\ell(\mathbb R^6)$ with $k\geq 6$ and $\ell>3/2+2s+\gamma.$ Then 
 the   Cauchy problem \eqref{3} admits a global solution $f\in L^\infty\inner{[0,+\infty[;~H^k_\ell(\mathbb R^6)}$, provided $\norm{f_0}_{H^k_\ell(\mathbb R^6)}$ is small enough. Moreover
a constant $C$ exists such that   
 \begin{eqnarray*}
 	\sup_{t\geq 0}\norm{f(t)}_{H^k_\ell}+\sum_{\abs\alpha+\abs\beta\leq k}  \Big(\int_0^{+\infty}\Big(\int_{\mathbb R^3} \normm{\comi v^\ell\partial_{x}^\alpha\partial_v^\beta f(t)}^2dx\Big)dt\Big)^{1/2} \leq 
 C \norm{f_0}_{H^k_\ell}.
 \end{eqnarray*}
 Recall $\normm{\cdot}$ is defined in \eqref{trnor}.
\end{theorem}

In this work we  will improve the Sobolev regularity at positive time  in the framework of Gevrey class. 

   
   \begin{definition}
   	Let $\mu\geq 1$ and we denote by $G^{\mu}$ the  space of all the $C^\infty$ functions $u(x,v)$ satisfying that a constant $C$ exists such that 
   \begin{eqnarray*}
 \forall\ \abs{\alpha}+\abs\beta\geq 0,\quad   \norm{	\partial_x^\alpha\partial_v^\beta u}_{L^2(\mathbb R_{x,v}^6)} \leq C^{\abs\alpha+\abs\beta+1}\big(\inner{\abs{\alpha}+\abs{\beta}}!\big)^\mu.
   \end{eqnarray*}
   Here $\mu$ is called the Gevrey index. We can also define an   anisotropic Gevrey space $G^{\mu_1,\mu_2}, \mu_j\geq1,$ which is consist of all  the  $C^\infty$ functions $u(x,v)$ satisfying that a constant $C$ exists such that 
   \begin{eqnarray*}
 \forall\ \abs{\alpha}+\abs\beta \geq 0,\quad   \norm{	\partial_x^\alpha\partial_v^\beta u}_{L^2(\mathbb R_{x,v}^6)} \leq C^{\abs\alpha+\abs\beta+1}\inner{\abs{\alpha}!}^{\mu_1}\inner{\abs{\beta}!}^{\mu_2}.
   \end{eqnarray*}
   \end{definition} 
   
   Before stating our main result we provide a representation  of the triple norm $\normm{\cdot}$ defined by \eqref{trnor} in term of a pseudo-differential operator.  Precisely we have (see Lemma \ref{lemdiff} in the next section)
   \begin{eqnarray*}
 \normm{u }^2  
  \approx  \norm{(a^{1/2})^w u}_{L^2(\mathbb R_v^3)}^2,  
\end{eqnarray*}
where $(a^{1/2})^w$ stands for the Weyl quantization with symbol $a^{1/2}.$   The definition of $a^{1/2}$ as well as some basic facts  on the symbolic calculus will be given in Subsection  \ref{subsec21} below  and Appendix \ref{secapp}.

 \begin{theorem}\label{maith} Assume that the cross-section satisfies \eqref{kern} and \eqref{angu}  with $0 < s < 1$  and $\gamma \geq 0.$
 	Let  $f\in L^\infty\inner{[0,+\infty];~H^2(\mathbb R^6)}$  be any solution to \eqref{3} such that 
 	\begin{equation}
 	\label{smacon}	
 	\sup_{t\geq 0}\norm{f(t)}_{H^2(\mathbb R^6)}+ \sum_{\abs\alpha+\abs\beta\leq 2}\inner{\int_0^{+\infty}  \norm{(a^{1/2})^w  \partial_x^\alpha\partial_v^\beta f(t)}_{L^2(\mathbb R^6)}^2dt}^{1/2} \leq 
 \epsilon_0
 	\end{equation}
 	for some constant $\epsilon_0>0.$  Suppose  $\epsilon_0$ is small enough.  Then
      $f(t)\in G^{\frac{1+2s}{2s}}$ for all $t>0.$   Moreover there exists a constant $C\geq 1$ depending only on  $s,\gamma$ and $\epsilon_0$ above such that for any multi-indices $\alpha$ and $\beta$ with $\abs\alpha+\abs\beta\geq0,$
  \begin{eqnarray*}
   	\sup_{t>0} \phi(t)^{\frac{1+2s}{2s}\inner{\abs\alpha+\abs\beta}}\norm{\partial_{x}^\alpha\partial_v^\beta f(t)}_{L^2(\mathbb R^6)}\leq  C^{\abs\alpha+\abs\beta+1}\big(\inner{\abs{\alpha}+\abs{\beta}}!\big)^{\frac{1+2s}{2s}}
  \end{eqnarray*}  
with $\phi(t)\stackrel{\textrm{def}}{=}\min\big\{t,  1\big\},$   or equivalently, 
  \begin{eqnarray*}
  	\sup_{t>0}\big\|e^{ c_0\phi(t)^{ \frac{1+2s}{2s}}\inner{- \Delta_{x,v}}^{\frac{s}{1+2s}}}f(t)\big\|_{ L^2(\mathbb R^6)}<+\infty
  \end{eqnarray*}
  for some constant $c_0>0.$
 \end{theorem} 
   
      \begin{remark}
    Theorem \ref{thex} guarantees the existence of solutions    satisfying  the assumption listed in Theorem \ref{maith}.   
   \end{remark}

    \begin{remark} 
   Theorem \ref{maith} provides  an explicit dependence of the Gevrey semi-norms  on the short time $0<t\leq 1.$   If the solutions admit additionally some kind of  decay for long  time     then the dependence on the large time  is also variable. Precisely, replacing the estimate  \eqref{smacon}  by the following  
   \begin{multline*}
 	\sup_{t\geq 0}\sum_{\abs\alpha+\abs\beta\leq 2}  \comi t^{\frac{1+2s}{2s}\inner{\abs\alpha+\abs\beta}}\norm{\partial_{x}^\alpha\partial_v^\beta f(t)}_{L^2(\mathbb R^6)}\\
 	+\sum_{\abs\alpha+\abs\beta\leq 2}\inner{\int_0^{+\infty} \comi t^{\frac{1+2s}{s} (\abs\alpha+\abs\beta)} \norm{(a^{1/2})^w\partial_{x}^\alpha\partial_v^\beta f(t)}_{L^2(\mathbb R^6)}^2dt}^{1/2} \leq 
 \epsilon_0,
 	\end{multline*}
 	we  can obtain that  
 	 \begin{eqnarray*}
  	\sup_{t>0}\big\|e^{ c_0t^{ \frac{1+2s}{2s}}\inner{- \Delta_{x,v}}^{\frac{s}{1+2s}}}f(t)\big\|_{ L^2(\mathbb R^6)}<+\infty
  \end{eqnarray*}
  for some constant  $c_0>0,$ provided the $\epsilon_0$ above is small enough.  The argument is quite similar as  that for proving  Theorem  \ref{maith} with slight modification, so we leave it to the interested readers.  Note the above estimate  means  we have polynomial decay to the equilibrium.  And the existence of solutions with strong exponential decay is obtained by Gressman-Strain  \cite{MR2784329}. 
     \end{remark}

     \begin{remark} The Gevrey index $(1+2s)/2s$ is just the same as that obtained by \cite{LERNER2015459} for the Kac equation, the one-dimensional model of Boltzmann equation.  This index is deduced from the sharp subelliptic estimate in spatial variable.    But we don't know wether or not the Gevrey index is optimal.  In fact consider the following generalized Kolmogorov equation,  which can be seen as a  simplified model of \eqref{3} if ignoring at moment the nonlinear term on the right-hand side, 
     \begin{eqnarray*}
\left\{
\begin{aligned}
&	\partial_t f+v\cdot\partial_x f+  (-\Delta_v )^{s}   f=0,\\
&f|_{t=0}=f_0.
	\end{aligned}
\right.
\end{eqnarray*}
A simple application of Fourier analysis shows the solution $f$ to the generalized Kolmogorov equation   has an explicit representation and satisfies 
\begin{eqnarray*}
	e^{c\inner{t(-\Delta_v)^{s}+t^{2s+1}(-\Delta_x)^{s}}} f \in L^2(\mathbb R^6)
\end{eqnarray*}
for some constant $c>0.$ This yields  $f\in G^{1/2s};$  we refer to \cite{MR2523694} for more detailed analysis on the model equation.  So it remains   interesting to verify  wether or not we can achieve the Gevrey index $1/2s,$ which seems to be optimal,  in space-velocity or  only velocity variable. 
     \end{remark}
     
     \begin{remark}
       Here we consider the Gevrey class regularization of solutions with mild regularity.    It is natural to ask what is the minimal regularity  required to boot the regularization procedure.    And  it is more interesting to ask the regularization effect of weak solutions  satisfying  only some kind of physical conditions such as finite mass, energy and  entropy.   
     \end{remark}
     

    \medskip
\noindent {\bf Notations}.  If no confusion occurs we will use $L^2$ to stand for the function space $L^2(\mathbb R_{x,v}^6)$, and  use $\norm{\cdot}_{L^2}$ and $\inner{\cdot,\ \cdot}_{L^2}$ to denote the norm and inner product of  $L^2=L^2(\mathbb R_{x,v}^6).$  We will also use the notations $\norm{\cdot}_{L^2(\mathbb R_v^3)}$ and  $\inner{\cdot,\ \cdot}_{L^2(\mathbb R_v^3)}$  when the variables are specified.   Similarly for $H^k$ and $H^k_\ell.$

Let $\xi$ and $\eta$ be the dual variables of $x$ and $v$ respectively. We denote by $\hat u(\xi) $  the (partial) Fourier transform in $x$ variable and denote by $q(D_x)$ a Fourier multiplier in $x$ variable with symbol $q(\xi)$,  that is, 
\begin{eqnarray*}
\widehat{q(D_x) u}(\xi)=q(\xi)\hat u(\xi). 
\end{eqnarray*}
Similarly we can define $q(D_v)$, a Fourier multiplier in $v$ variable with symbol $q(\eta).$ In particular
let $\comi{D_x}^\tau$ be the Fourier multiplier with symbol $\comi{\xi}^\tau$, recalling $\comi\xi=(1+\abs\xi^2)^{1/2}.$   Similarly we can define   $\comi{D_v}^\tau.$
Given a function $p(x,v;\xi,\eta)$ we denote by $p^w$ and $p^{\rm Wick}$ the Weyl and Wick quantizations respectively  with symbol $(v,\eta)\rightarrow p(x,v; \xi,\eta),$ considering $x$ and $\xi$ as parameters. The precise definition of Weyl and Wick quantizations is given in Appendix \ref{secapp}.   

Given two operators $T_1$ and $T_2$ we denote by $[T_1, T_2]$ the commutator between $T_1$ and $T_2$, that is,
\begin{eqnarray*}
	[T_1,  T_2]=T_1T_2-T_2T_1.
\end{eqnarray*}
  We say $T_1$ commutes with $T_2$ if  $[T_1,  T_2]=0.$

\section{Subelliptic estimate for Cauchy problem}
Let $P$ be the linearized Boltzmann operator given in \eqref{linbol}.   In this part we will derive a subelliptic estimate for the linear Cauchy problem of  Boltzmann operator $P.$ To do so we need the symbolic calculus developed in \cite{AHL}.

\subsection{Some facts on symbolic calculus}\label{subsec21}

Here we recall without proof  some  facts on the symbolic calculus obtained in  \cite{AHL} for the  collision  operator $\mathcal L$ defined in \eqref{linearbol}.  To do so we   introduce some notations for the phase space analysis and list in  Appendix \ref{secapp}  the basic properties of the quantization of symbols, and  we refer  to \cite{MR2599384}  for  the comprehensive discussion.

 Let $\eta\in\mathbb R^3$ be the dual variable of velocity $v$, and throughout the paper we let $\tilde a$ be   defined by  
 \begin{equation}\label{defofa}
  \tilde{a}(v,\eta) =\comi{v}^{\gamma} (1+ |\eta|^{2} + |v\wedge \eta|^{2} +\abs{v}^{2} )^s, \quad (v,\eta) \in \mathbb R^6,
  \end{equation}
where $\gamma$ and $s$  are the numbers given respectively  in \eqref{kern} and \eqref{angu}, and $v\wedge \eta$ stands for the cross product of two vectors $v$ and $\eta$.      Direct computation shows $\tilde a$ is an admissible weight for the flat metric $\abs{dv}^2+\abs{d\eta}^2$ (cf. \cite{AHL} for the  verification in detail),  and we can consider the Weyl quantization $p^{w}$ and Wick quantization $p^{\rm Wick}$ of a symbol $p$ lying in the class $S(\tilde a,  \abs{dv}^2+\abs{d\eta}^2);$  see   Appendix \ref{secapp}  for the definition of admissible weights,    symbol class  and its quantization.

 \begin{proposition}[Proposition 1.4 and Lemma 4.2 in \cite{AHL}] \label{estaa} Assume that the cross-section satisfies \eqref{kern} and \eqref{angu}  with $0 < s < 1$  and $\gamma >-3.$   Let $\mathcal L$ be the linearized collision operator defined by \eqref{linearbol}.   Then  we can write $$\mathcal L = -a^w - \mathcal R,$$ 
 where  $a^w$ stands for the Weyl quantization of  the symbol $a,$  with the properties listed below fulfilled. 
\begin{enumerate}[align=right, leftmargin=*,  label=(\roman*)]
\item
We have
$  a, \tilde a \in S( \tilde{a} , \abs{dv}^2+\abs{d\eta}^2)$, 
and moreover there exists a  positive constant $C\geq 1$ such that
 $C^{-1} \tilde{a}(v,\eta) \leq a(v, \eta) \leq  C \tilde{a}(v,\eta)$ for all $(v,\eta)\in\mathbb R^6$.
  
\item For any $\eps > 0$ there exists a constant $C_\eps$ such that
$$
  \norm{\mathcal R f} \leq \eps\norm{a^w f } + C_\eps \norm{\comi{v}^{\gamma+ 2s} f}.
$$
\item  The operators $ a^w$ and $\big(a^{1/2}\big)^w$  are invertible    in $L^2$ and  their inverses can be written as
\[
    \inner{a^w}^{-1} =H_1 \inner{a^{-1}}^w=\inner{a^{-1}}^wH_2
\]
and 
\[
   \big[\big(a^{1/2}\big)^w\big]^{-1}= G_1
   \big(a^{-1/2}\big)^w= \big(a^{-1/2}\big)^wG_2,
\]
with  $H_j, G_j$  the  bounded operators
in $L^2$.
   \end{enumerate}
   \end{proposition}
   
        \begin{remark}
      To simplify the notation  we write $a$ here instead of $a_K$  defined in  \cite[Proposition 1.4]{AHL} with $K$ a large positive number. Accordingly $\mathcal R=\mathcal K-K\comi v^{2s+\gamma}$ with $\mathcal K$ given in  \cite[Proposition 1.4]{AHL} or defined precisely  in  
  \cite[eq. (53)]{AHL}.    
     \end{remark}
       
  The symbolic calculus above enables us to get the  exact diffusion property of the collision operator and provides the following  representation of the triple norm defined by \eqref{trnor} in terms of $(a^{1/2})^w$.

 \begin{lemma}\label{lemdiff}
Assume that the cross-section satisfies \eqref{kern} and \eqref{angu}  with $0 < s < 1$  and $\gamma \geq 0.$  Then  for all $l \in \mathbb R$
with $l \leq \gamma/2+s$ and  for any  suitable function $u$ we have 
\begin{eqnarray*}
 \normm{u }^2  
 \approx -\inner{\mathcal L u, u}_{L^2(\mathbb R_v^3)} + \norm{\comi v^{l}f}_{L^2(\mathbb R_v^3)}^2  \approx  \norm{(a^{1/2})^w u}_{L^2(\mathbb R_v^3)}^2,
\end{eqnarray*}
where in the first  equivalence the constant  depends only  on $l $.
\end{lemma}

\begin{proof}
The first  equivalence is obtained by \cite[Theorem 1.1]{AMUXY1} 	and the second one follows from  \cite[Lemma 4.6 and Lemma 4.8]{AHL}.
\end{proof}

As a result of Lemma \ref{lemdiff}  the following upper bound for trilinear term (see   \cite[Theorem 1.2]{AlexMorUkaiXuYang})  
\begin{equation*}
|(\Gamma(g,h),\psi)_{L^{2}(\mathbb{R}^{3}_{v})} |\leq  C \norm{g}_{L^{2}(\mathbb{R}^{3}_{v})}\normm h\, \normm\psi
\end{equation*}
can be re-written as  
\begin{equation}\label{orup}
|(\Gamma(g,h),\psi)_{L^{2}(\mathbb{R}^{3}_{v})} |\leq  C \norm{g}_{L^{2}(\mathbb{R}^{3}_{v})}\norm{ (a^{1/2})^w h}_{L^{2}(\mathbb{R}^{3}_{v})}\norm{(a^{1/2})^w \psi}_{L^{2}(\mathbb{R}^{3}_{v})}.
\end{equation}
This gives for suitable function $g, h$ and $\psi,$
\begin{equation}\label{tripleest}
|(\Gamma(g,h),\psi)_{L^{2}} |\leq  C \sum_{\abs{\beta}\leq 2}\norm{\partial_x^\beta g}_{L^{2}}\norm{ (a^{1/2})^w h}_{L^{2}}\norm{(a^{1/2})^w \psi}_{L^{2}},
\end{equation}
\begin{equation}\label{tripleest+}
|(\Gamma(g,h),\psi)_{L^{2}} |\leq  C \norm{g}_{L^{2}}\Big(\sum_{\abs{\beta}\leq 2} \norm{ (a^{1/2})^w \partial_x^\beta  h}_{L^{2}}\Big)\norm{(a^{1/2})^w \psi}_{L^{2}}
\end{equation}
and 
\begin{equation}\label{tripleest++++}
|(\Gamma(g,h),\psi)_{L^{2}} |\leq  C \norm{g}_{L^{2}}\norm{ (a^{1/2})^w   h}_{L^{2}}\Big(\sum_{\abs{\beta}\leq 2} \norm{\partial_x^\beta (a^{1/2})^w \psi}_{L^{2}}\Big),
\end{equation}
 recalling $L^2=L^2(\mathbb R^6_{x,v}).$   Moreover we have
   the following inequality: for any $u, w\in H^1(\mathbb R_x^3),$
\begin{equation}\label{sobeq}
	\norm{uw}_{L^2(\mathbb R_x^3)}\leq C \Big(\sum_{\abs\beta\leq 1}\norm{\partial_x^\beta u}_{L^2(\mathbb R_x^3)}\Big) \sum_{\abs\beta\leq 1}\norm{\partial_x^\beta w}_{L^2(\mathbb R_x^3)}.
\end{equation}
To see this observe \begin{eqnarray*}
 \norm{uw}_{L^2(\mathbb R_x^3)} \leq \norm{u}_{L^4(\mathbb R_x^3)}\norm{w}_{L^4(\mathbb R_x^3)} 	\end{eqnarray*}
	and  Gagliardo-Nirenberg-Sobolev inequality gives
	\begin{eqnarray*}
	\norm{u}_{L^4(\mathbb R_x^3)}\leq \norm{u}_{L^2(\mathbb R_x^3)}^{1/4}\norm{u}_{L^6(\mathbb R_x^3)}^{3/4}\leq C\norm{u}_{L^2(\mathbb R_x^3)}^{1/4}\norm{\partial_xu}_{L^2(\mathbb R_x^3)}^{3/4}\leq C\sum_{\abs\beta\leq 1}\norm{\partial_x^\beta u}_{L^2(\mathbb R_x^3)}.
	\end{eqnarray*}
	Similar for $\norm{w}_{L^4(\mathbb R_x^3)}.$  Then \eqref{sobeq} follows.  Combining \eqref{sobeq} and \eqref{orup} implies 
	\begin{equation}
	\label{newesti}	
	|(\Gamma(g,h),\psi)_{L^{2}} |\leq  C \Big(\sum_{\abs{\beta}\leq 1} \norm{\partial_x^\beta g}_{L^{2}}\Big)\norm{ (a^{1/2})^w   h}_{L^{2}} \Big(\sum_{\abs{\beta}\leq 1} \norm{\partial_x^\beta ( a^{1/2})^w \psi}_{L^{2}}\Big).
	\end{equation}

\subsection{Subelliptic estimate for the Cauchy problem of linear Boltzmann equation}
The non isotropic   triple norm  given in the previous subsection  is not enough for  the  Gevrey regularity,  since  we can't  get any regularity in spatial variable $x.$  When the time varies in the whole space,   the  sharp regularity in all variables  is obtained by \cite{AHL} using  the Fourier transform in time-space variable.    Here we will derive a  subelliptic estimate involving the initial data,  following the multiplier method  in \cite{AHL}.

 \begin{proposition}[Elliptic estimate in velocity]\label{prpellp}
 	Assume that the cross-section satisfies \eqref{kern} and \eqref{angu}  with $0 < s < 1$  and $\gamma \geq 0.$    Then there exists a  constant  $C\geq 1$ such that for any given $r\geq 1$  and   any function $u$ satisfying 
 that
 $Pu\in L^2 \inner{[0, 1]\times\mathbb  R^6}$ and that
 \begin{equation*}
 \qquad 	t^{r-{1\over2}}u(t)\in L^\infty \inner{[0, 1];  L^2}\  \textrm{and}\  \  t^{r}(a^{1/2})^w u(t) \in  L^2 \inner{[0, 1]\times\mathbb  R^6},
 \end{equation*}    
  we have, for any $0<t\leq 1,$
\begin{equation*}\label{subest+++}
\begin{aligned}
	&    t^{2r}\norm{u(t)}_{L^2}^2  + \int_0^1 t^{2r}\norm{(a^{1/2})^wu(t)}_{L^2}^2dt \\
	 \leq & C\int_0^1  t^{2r} \big|\inner{ P   u,\   u}_{L^2}\big| dt +r\,C    \int_0^1 t^{2r-1}\norm{u}_{L^2}^2dt.
	   \end{aligned}
\end{equation*}

 \end{proposition}

\begin{proof}
	By density we may assume $u$ is rapidly decreasing  on $\mathbb R^6.$   Using the second equivalence in Lemma \ref{lemdiff} as well as the fact that 
\begin{eqnarray*}
	-\inner{\mathcal L u, u}_{L^2}={\rm  Re} \inner{Pu, u}_{L^2}-\frac{1}{2}\frac{d}{dt}\norm{u}_{L^2}^2,
\end{eqnarray*}
we conclude   a small constant $0<c_1<1$ exists such that   
\begin{equation}\label{diffes}
	\frac{1}{2} \frac{d}{dt}\norm{u}_{L^2}^2+ c_1 \norm{(a^{1/2})^w u}_{L^2}^2 \leq \abs{(Pu,u)_{L^2}}+\left\| u\right\|_{L^2}^2.  
\end{equation}
Thus for any  $0<t\leq 1,$
\begin{eqnarray*}
	\begin{aligned}
  &\frac{1}{2}\frac{d}{dt} \inner{t^{2r}\norm{u}_{L^2}^2} +   c_1 t^{2r}\norm{(a^{1/2})^wu}_{L^2}^2  \\
	 \leq  &  t^{2r} \big|\inner{ P   u,\   u}_{L^2}\big|+ t^{2r}\norm{u}_{L^2}^2+rt^{2r-1}\norm{ u}_{L^2}^2.
\end{aligned}
\end{eqnarray*}
Integrating both side over the interval $[0,t]$ with any $0<t\leq 1$ and observing  $t^{r-1/2}u\in L^\infty\inner{[0, 1];\ L^2}$ which implies 
\begin{eqnarray*}
 \lim_{t\rightarrow 0}t^{2r}  \norm{ u }_{L^2} ^2 =0,
\end{eqnarray*} 
we obtain the  estimate as desired   for $v$ variable.  The proof is completed. 
\end{proof}

\begin{proposition}[Subelliptic estimate in space]
\label{prosub}
Assume that the cross-section satisfies \eqref{kern} and \eqref{angu}  with $0 < s < 1$  and $\gamma \geq 0.$   Then we can find a   constant $C\geq 1$   and a bounded operator   $\mathcal A $  in $L^2$ with the following properties 
 \begin{equation}
	\label{bd}
	\left\{
	\begin{aligned}
	&\norm{\mathcal A  u}_{L^2}\leq C_{s,\gamma} \norm{u}_{L^2},\\
	&\big[\mathcal A , \ q(D_x) \big]=0,\\
	& \norm{\big[\mathcal A , \  (a^{1/2})^w\big]u}_{L^2}\leq C_{s,\gamma} \norm{(a^{1/2})^w u}_{L^2},
	\end{aligned}
	\right.
\end{equation}
fulfilled for     some  constant $C_{s,\gamma}$   depending only on $s$ and $\gamma$ and  for  any  Fourier multiplier  $q(D_x)$ in  only $x$ variable,     such that   
  for any given $r\geq 1$  the following two estimates hold.
 \begin{enumerate}[fullwidth, itemindent=0em, label=(\roman*)]
\item For any function $u$ satisfying 
 that
 $Pu\in L^2 \inner{[0, 1]\times\mathbb  R^6}$ and that
 \begin{equation*}
  	t^{r-{1\over2}}u\in L^\infty \inner{[0, 1];  L^2}\ \textrm{and}\ \   t^{r}\comi{D_x}^{\frac{s}{1+2s}}u, \, t^{r}(a^{1/2})^w u \in  L^2 \inner{[0, 1]\times\mathbb  R^6},
 \end{equation*}    
  we have, for any $0<t\leq 1,$
\begin{eqnarray}\label{subest}
\begin{aligned}
	&    t^{2r}\norm{u(t)}_{L^2}^2 +  \int_0^1 t^{2r}\norm{\comi{D_x}^{\frac{s}{1+2s}} u(t)}_{L^2}^2dt+ \int_0^1 t^{2r}\norm{(a^{1/2})^wu(t)}_{L^2}^2dt \\
	& \leq  C\int_0^1  t^{2r} \big|\inner{ P   u,\   u}_{L^2}\big| dt +C \int_0^1 t^{2r} \big|\inner{ P   u,\   \mathcal A  u}_{L^2}  \big| dt +r\,C    \int_0^1 t^{2r-1}\norm{u}_{L^2}^2dt.
	   \end{aligned}
\end{eqnarray}
\item  For any function $u$ satisfying 
 that
 $Pu\in L^2 \inner{[1, +\infty[\times\mathbb  R^6}$ and that
 \begin{equation*}
 \qquad  u \in L^\infty \inner{[1, +\infty[;  L^2}\ \textrm{and}\   \comi{D_x}^{\frac{s}{1+2s}}u, \   (a^{1/2})^w u \in  L^2 \inner{[1, +\infty[\times\mathbb  R^6},
 \end{equation*}    we have, for any $t\geq 1,$
\begin{equation}\label{1infty}
\begin{aligned}
	&   \norm{u(t)}_{L^2}^2 +  \int_1^{+\infty}   \norm{\comi{D_x}^{\frac{s}{1+2s}} u(t)}_{L^2}^2dt + \int_1^{+\infty}  \norm{(a^{1/2})^wu(t)}_{L^2}^2dt \\
	\leq & \ \norm{u(1)}_{L^2}^2+ C\int_1^{+\infty}    \big|\inner{ P   u, \   u}_{L^2}\big| dt +C \int_1^{+\infty}    \big|\inner{ P   u,  \  \mathcal A  u}_{L^2}  \big| dt\\
	& +C   \int_1^{+\infty} \norm{u}_{L^2}^2dt.
	 \end{aligned}
\end{equation} 
\end{enumerate}
Note the constant $C$  in \eqref{subest} is independent of $r.$

\end{proposition}

\begin{proof} 
 We adopt the idea used for  proving  \cite[Lemma 4.12]{AHL}. Let $u$ be an arbitrarily given function satisfying the assumption above.    By density we may assume $u$ is rapidly decreasing  on $\mathbb R^6.$     Recall  $\xi $ is the dual variable of $x$ and  $\hat  u (\xi, v)$ is the partial Fourier transform of $u(x,v)$ with respect to $x.$
Then we have   \begin{eqnarray*}
  	\widehat{Pu}(t,\xi, v) =\inner{\partial_t+iv\cdot\xi-\mathcal L}\hat u (t,\xi, v).
  \end{eqnarray*}

   Let $\lambda^{\rm Wick}$ be the Wick quantization (see   Appendix \ref{secapp} for the definition of Wick quantization)  of symbol $\lambda$ with       
\begin{equation*} \label{cutcut}
\lambda(v,\eta) =
\lambda_{\xi}(v,\eta)=\frac{  d(v,\eta)}{ \tilde a(v,\xi)^{\frac{2s}{1+2s}}}\chi\bigg( \frac{  \tilde a(v,\eta)}{  \tilde a(v,\xi)^{\frac{1}{1+2s}}}\bigg),
\end{equation*}
where $\chi\in C_0^\infty(\mathbb R;~[0,1])$ such that
$\chi=1$ on
$[-1,1]$ and supp~$\chi \subset[-2,2]$ and 
 \begin{eqnarray*}
d(v,\eta) = \comii v^{\gamma}\inner{1+\abs v^2+\abs\xi^2+\abs{v\wedge \xi}^2}
^{s-1}\Big(\xi\cdot\eta+(v\wedge \xi)\cdot(v\wedge \eta)\Big).
\end{eqnarray*}
Recall $\tilde a $
 is defined in \eqref{defofa}.  Direct computation shows 
\begin{eqnarray*}
	\lambda\in   S\big(1, \abs{dv}^2+\abs{d\eta}^2\big)
\end{eqnarray*}
uniformly with respect to $\xi.$
As a result $\lambda^{\rm Wick}$ is a bounded operator in $L^2:$
\begin{equation}\label{bdsym}
\forall\ u\in  L^2, \quad 	\norm{\lambda^{\rm Wick}u}_{L^2}\leq C_{\gamma,s} \norm{u}_{L^2}
\end{equation}
for some constant $C_{s,\gamma}$ depending only on $s$ and $\gamma.$ 
The advantage of $\lambda^{\rm Wick}$  lies in the fact that the interaction between $\lambda^{\rm Wick}$  and the transport part will yield the regularity in $x.$  Precisely,  observing    $ v\cdot\xi=( v\cdot\xi)^{\rm Wick}$ with $\xi$ a parameter,  we use the relationship \eqref{11082406} in   Appendix \ref{secapp} to get 
\begin{eqnarray*}
	{\rm Re} \inner{i (v\cdot\xi) \hat u,\   \lambda^{\rm Wick}\hat u}_{L^2(\mathbb R_v^3)} =\frac{1}{2} \inner{ \big\{\lambda, v\cdot\xi\big\}^{\rm Wick} \hat u,\  \hat u}_{L^2(\mathbb R_v^3)},
\end{eqnarray*}
 with $\set{\cdot, \cdot}$ the Poisson
bracket defined by \eqref{11051505}.
Moreover  using  the positivity property of Wick quantization (see Appendix \ref{subwick})  yields   
\begin{eqnarray*}
	 &&\frac{1}{2}\inner{ \big\{\lambda, v\cdot\xi\big\}^{\rm Wick} \hat u,\  \hat u}_{L^2(\mathbb R_v^3)} \\
	 &\geq & \frac{1}{2}\inner{ \big(\tilde a(v,\xi)^{\frac{1}{1+2s}}\big)^{\rm Wick} \hat u,\  \hat u}_{L^2(\mathbb R_v^3)}   - C_1\inner{ \big(\tilde a(v,\eta)\big)^{\rm Wick} \hat u,\  \hat u}_{L^2(\mathbb R_v^3)}\\
	 &\geq& c_2\norm{\comii v^{\gamma/(2+4s)}\comi{\xi}^{s/(1+2s)}\hat u}_{L^2(\mathbb R_v^3)}^2	   - C_2 \norm{  (a^{1/2})^w \hat u}_{L^2(\mathbb R_v^3)}^2, 
\end{eqnarray*}
where  and throughout the proof  $0<c_2<1$ and  we use $ C_j, j\geq 1,$ to denote different    constants depending on $s$ and $\gamma$;   see    
 \cite[Lemma 4.12]{AHL} for     proving the above inequalities in detail.    Combining these   estimates we conclude, using the fact that $\gamma\geq 0, $ 
\begin{equation}\label{fie}
	c_2\norm{\comi{\xi}^{s/(1+2s)}\hat u}_{L^2(\mathbb R_v^3)}^2\leq 	{\rm Re} \inner{i (v\cdot\xi) \hat u,\   \lambda^{\rm Wick}\hat u}_{L^2(\mathbb R_v^3)}+    C_2 \norm{  (a^{1/2})^w \hat u}_{L^2(\mathbb R_v^3)}^2.
\end{equation}
As for the first term on the right-hand side we have  
\begin{eqnarray*}
	&&	{\rm Re} \inner{i (v\cdot\xi) \hat u,\   \lambda^{\rm Wick}\hat u}_{L^2(\mathbb R_v^3)}\\
	&= &{\rm Re} \inner{\widehat{P   u},\   \lambda^{\rm Wick}\hat u}_{L^2(\mathbb R_v^3)}-{\rm Re} \inner{\partial_t \hat{   u},\   \lambda^{\rm Wick}\hat u}_{L^2(\mathbb R_v^3)} +{\rm Re} \inner{\mathcal L \hat{   u},\   \lambda^{\rm Wick}\hat u}_{L^2(\mathbb R_v^3)}\\
		&\leq &	{\rm Re} \inner{\widehat{P   u},\   \lambda^{\rm Wick}\hat u}_{L^2(\mathbb R_v^3)}-\frac{1}{2}\frac{d}{dt} \inner{ \hat{   u},\   \lambda^{\rm Wick}\hat u}_{L^2(\mathbb R_v^3)}+C_3  \norm{a^{1/2})^w\hat u}_{L^2(\mathbb R_v^3)}^2,
\end{eqnarray*}
the last line holding because $\lambda^{\rm Wick}$ is self-adjoint in $L^2(\mathbb R^3_v)$ and moreover  using the assertions  $(i)$ and $(iii)$ in Proposition \ref{estaa} gives
\begin{eqnarray*}
	&&\big| {\rm Re}  \inner{\mathcal L \hat{   u},\   \lambda^{\rm Wick}\hat u}_{L^2(\mathbb R_v^3)}\big|\\
	&=&\big|{\rm Re}\big( \underbrace{ \big[ (a^{1/2})^w\big]^{-1}  \mathcal L  \big[ (a^{1/2})^w\big]^{-1} }_{\textrm{bounded operator}} (a^{1/2})^w \hat{   u},\   \underbrace{  (a^{1/2})^w \lambda^{\rm Wick}\big[ (a^{1/2})^w\big]^{-1}  }_{\textrm{bounded operator}} (a^{1/2})^w\hat u\big)_{L^2(\mathbb R_v^3)}\big|\\
	&\leq& C_3 \norm{a^{1/2})^w\hat u}_{L^2(\mathbb R_v^3)}^2.
\end{eqnarray*}
Combining the above estimate with \eqref{fie}  we obtain
\begin{multline*}
	c_2\norm{\comi{\xi}^{\frac{s}{1+2s}}\hat u}_{L^2(\mathbb R_v^3)}^2\\
	 \leq {\rm Re}  \inner{\widehat{P   u},\   \lambda^{\rm Wick}\hat u}_{L^2(\mathbb R_v^3)}-\frac{1}{2}\frac{d}{dt} \inner{ \hat{   u},\   \lambda^{\rm Wick}\hat u}_{L^2(\mathbb R_v^3)}+C_4  \norm{(a^{1/ 2})^w\hat u}_{L^2(\mathbb R_v^3)}^2.
\end{multline*}
Define the operator $\mathcal A $  by
\begin{eqnarray*}
\widehat{\mathcal A  u}(\xi, v)=\lambda^{\rm Wick} \hat u(\xi,v).
\end{eqnarray*}
Then it follows from the above inequality and 
 Plancherel formula  that  
\begin{equation}\label{diff31}
	c_2\norm{\comi{D_x}^{\frac{s}{1+2s}} u}_{L^2}^2 \leq \big|\inner{ P   u,\   \mathcal A  u}_{L^2}\big|-\frac{1}{2}\frac{d}{dt} \inner{ u,\   \mathcal A  u}_{L^2}+C_4  \norm{(a^{1/ 2})^w  u}_{L^2}^2
\end{equation}
and
 \begin{equation}\label{uppfora}
	\norm{\mathcal A  u}_{L^2} = \norm{ \lambda^{\rm Wick} \hat u}_{L^2(\mathbb R_{\xi,v}^6)}\leq C_{s,\gamma} \norm{ \hat u}_{L^2(\mathbb R_{\xi,v}^6)} =  C_{s,\gamma} \norm{  u}_{L^2}. 
	 \end{equation} 
Now we choose such a  $N$  that 
\begin{equation}\label{chn}
	N =\max\big\{ 2C_4/c_1, 2C_{\gamma, s}\big\} +4
\end{equation}
with $C_4$ given in \eqref{diff31},  $c_1$ the number in \eqref{diffes} and $ C_{\gamma, s}$ the constant in \eqref{bdsym}.  Then   we  multiply both sides of \eqref{diffes} by $N$ and  then  add to \eqref{diff31};  this gives
\begin{equation}\label{mt}
\begin{aligned}
	& \frac{N}{2}\frac{d}{dt}\norm{u}_{L^2}^2+ c_2\norm{\comi{D_x}^{\frac{s}{1+2s}} u}_{L^2}^2+\frac{Nc_1}{2}\norm{(a^{1/2})^wu}_{L^2}^2 \\
	 \leq & N \big|\inner{ P   u,\   u}_{L^2}\big|+\big|\inner{ P   u,\   \mathcal A  u}_{L^2}\big|-\frac{1}{2}\frac{d}{dt} \inner{ u,\   \mathcal A  u}_{L^2}+N  \norm{ u}_{L^2}^2 ,
\end{aligned}
\end{equation}
and thus
\begin{eqnarray*}\label{timein}
	&& \frac{N}{2}\frac{d}{dt} \inner{t^{2r}\norm{u}_{L^2}^2}+ c_2t^{2r}\norm{\comi{D_x}^{\frac{s}{1+2s}} u}_{L^2}^2+ \frac{Nc_1}{2}t^{2r}\norm{(a^{1/2})^wu}_{L^2}^2 \\
	& \leq&  N t^{2r} \big|\inner{ P   u,\   u}_{L^2}\big|+ t^{2r} \big|\inner{ P   u,\   \mathcal A  u}_{L^2}\big|-\frac{1}{2}\frac{d}{dt} \Big[t^{2r} \inner{ u,\   \mathcal A  u}_{L^2}\Big]+N t^{2r} \norm{ u}_{L^2}^2\\
&&	+Nrt^{2r-1}\norm{u}_{L^2}^2+rt^{2r-1}\inner{ u,\   \mathcal A  u}_{L^2}.
\end{eqnarray*}
Integrating both side over the interval $[0,t]$ for any $0<t\leq 1$ and observing $t^{r-1/2}u\in L^\infty\inner{[0,1];\ L^2}$ which along with \eqref{uppfora} implies 
\begin{eqnarray*}
0\leq \lim_{t\rightarrow 0}t^{2r} \abs{\inner{ u,\   \mathcal A  u}_{L^2}	}\leq C_{\gamma,s}\lim_{t\rightarrow 0}t^{2r}  \norm{ u }_{L^2} ^2 =0,
\end{eqnarray*} 
we obtain, using \eqref{uppfora} again, 
\begin{eqnarray*}
	&& \frac{N}{2}  t^{2r}\norm{u(t)}_{L^2}^2 + c_2\int_0^t t^{2r}\norm{\comi{D_x}^{\frac{s}{1+2s}} u}_{L^2}^2dt+ \frac{Nc_1}{2}\int_0^t t^{2r}\norm{(a^{1/2})^wu}_{L^2}^2dt \\
	& \leq&  N\int_0^t  t^{2r} \Big|\inner{ P   u,  u}_{L^2}\Big| dt +\int_0^t t^{2r} \big|\inner{ P   u,   \mathcal A  u}_{L^2} \big|dt+\frac{C_{\gamma,s}}{2}  t^{2r} \norm{ u(t)}_{L^2}^2\\
	&&+N \int_0^t t^{2r} \norm{ u}_{L^2}^2dt+\inner{N+C_{\gamma,s}} r \int_0^t t^{2r-1}\norm{u}_{L^2}^2dt
	\end{eqnarray*}
for any $0<t\leq 1.$	  Thus, observing $r\geq 1$ and $C_{\gamma, s}\leq N/2$ due to \eqref{chn},
	\begin{eqnarray*}
	&& \frac{N}{4}  t^{2r}\norm{u}_{L^2}^2 + c_2\int_0^1 t^{2r}\norm{\comi{D_x}^{\frac{s}{1+2s}} u(t)}_{L^2}^2dt+ \frac{Nc_1}{2}\int_0^1 t^{2r}\norm{(a^{1/2})^wu(t)}_{L^2}^2dt \\
	& \leq & N\int_0^1  t^{2r} \big|\inner{ P   u,  u}_{L^2}\big| dt +\int_0^1 t^{2r} \big|\inner{ P   u, \mathcal A  u}_{L^2}  \big| dt+3N  r \int_0^1 t^{2r-1}\norm{u}_{L^2}^2dt.
	\end{eqnarray*}
The above inequality holds for all $0< t\leq 1.$   
Thus the desired \eqref{subest} follows if we choose   $C=24 \max\set{\frac{1}{c_1},  \frac{2N}{c_2}}.$  Similarly   integrating \eqref{mt} over $[1, t[$ for any $t>1,$
 we obtain the estimate \eqref{1infty}.

It remains to prove the assertions in \eqref{bd}, and  the first one follows from \eqref{uppfora}.
	    The second assertion in \eqref{bd} is obvious since the spatial variable $x$ is not involved in the symbol $\lambda.$  To prove the last assertion,  we only need work with  the $L^2(\mathbb R_{\xi, v}^6)$-norm by Plancherel formula.   The symbol of the commutator 
	 \begin{eqnarray*}
	 	\big[\lambda^{\rm Wick} , \  (a^{1/2})^w\big] 
	 \end{eqnarray*}
	 belongs to $S(\tilde a^{1/2}, \abs{dv}^2+\abs{d\eta}^2)$ since $\lambda\in S(1, \abs{dv}^2+\abs{d\eta}^2)$ uniformly for $\xi$ and $a^{1/2}\in S(\tilde a^{1/2}, \abs{dv}^2+\abs{d\eta}^2).$ As a result  we  can write 
	 \begin{eqnarray*}
	 	\big[\lambda^{\rm Wick}, \  (a^{1/2})^w\big] =\underbrace{	\big[\lambda^{\rm Wick}, \  (a^{1/2})^w\big] \big[(a^{1/2})^w\big]^{-1}}_{\textrm{bounded operator}} (a^{1/2})^w
	 \end{eqnarray*} 
	due to the conclusions $(i)$ and $ (iii)$ in Proposition \ref{estaa},  and thus  the third assertion in \eqref{bd} follows.  The proof is then completed. 
\end{proof}

\section{Gevrey regularity in spatial variable}
 This part  is devoted to proving the Gevrey smoothing effect in spatial variable $x,$ that is,  
\begin{theorem}\label{thm3.1}
Under the same assumption as in Theorem \ref{maith}, we can find 
  a positive constant $C_0$,  depending only on $s,\gamma$ and the constant $\epsilon_0$ in \eqref{smacon},   such that    	\begin{equation*}
	\forall\ \abs{\alpha}\geq 0,\quad	\sup_{t>0} \phi(t)^{\frac{1+2s}{2s}\abs\alpha}\norm{ \partial_{x}^{\alpha} f(t)}_{L^2} \leq C_{0}^{\abs\alpha+1}\inner{{\abs\alpha!}}^{\frac{1+2s}{2s}}.
	\end{equation*}
 Recall $\phi(t)=\min \set{t,1}.$
\end{theorem}

We will use induction to prove the above theorem, and the following proposition is crucial. 

\begin{proposition}\label{pro3.2} 
   Denote   $\kappa=(1+2s)/2s.$      Let  $f\in L^\infty\inner{[0,+\infty[;  H^2}$ be any solution  to the Cauchy problem  \eqref{3} satisfying the condition \eqref{smacon}.     Suppose additionally that there exists a positive constant $C_*\geq 1$,  depending only on $s,\gamma$ and the constant $\epsilon_0$ in \eqref{smacon} such that for any multi-index $\beta$ with  $2\leq \abs\beta\leq 3$ we have
   \begin{equation}\label{indass}
   \begin{aligned}
   &\sup_{t>0}  \phi(t)^{\kappa(\abs\beta-2)}\norm{ \partial_{x}^{\beta} f(t)}_{L^2}+\inner{\int_0^{+\infty} \phi(t)^{2\kappa(\abs\beta-2)} \norm{ \comi{D_x}^{\frac{s}{1+2s}}\partial_{x}^{\beta} f(t)}_{L^2}^2 dt}^{1/ 2}\\
	  	   &\qquad \qquad +\inner{\int_0^{+\infty}  \phi(t)^{2\kappa(\abs\beta-2)} \norm{ (a^{1/2})^w\partial_{x}^{\beta} f(t)}_{L^2}^2 dt}^{1/ 2}  \leq  
	  	  C_*.	
	\end{aligned}
   \end{equation}
Let $m\geq 5$ be an arbitrarily given integer.  Then we can find a positive constant $C_0 \geq C_*$,  depending only on $s,\gamma$ and  the constant $\epsilon_0$ in \eqref{smacon}    but independent of $m,$  such that if the following estimate 
  \begin{equation}\label{dayu3+}
 \begin{aligned}  	  
  &\sup_{t>0}  \phi(t)^{\kappa(\abs\beta-2)}\norm{ \partial_{x}^{\beta} f(t)}_{L^2}+\inner{\int_0^{+\infty} \phi(t)^{2\kappa(\abs\beta-2)} \norm{ \comi{D_x}^{\frac{s}{1+2s}}\partial_{x}^{\beta} f(t)}_{L^2}^2 dt}^{1\over2}\\
	  	   &\qquad\qquad\qquad+\inner{\int_0^{+\infty}  \phi(t)^{2\kappa(\abs\beta-2)} \norm{ (a^{1/2})^w\partial_{x}^{\beta} f(t)}_{L^2}^2 dt}^{1/ 2} \\
	  	\leq & \   
	  	  C_0^{ \abs\beta-3}\com{ (\abs\beta-4) !}^{\frac{1+2s}{2s}}	  	\end{aligned}
	\end{equation}
	holds for any  $\beta$ with $ 4\leq \abs\beta\leq m-1,
$   
then for any multi-index $\alpha$ with $\abs\alpha=m$ we have
	 $\phi(t)^{\kappa(m-2)} \partial_{x}^\alpha f \in L^\infty\inner{]0, +\infty[;\ L^2}$    and 
 $$\phi(t)^{\kappa(m-2)} \comi{D_{x}}^{\frac{s}{1+2s}}\partial_{x}^\alpha f, \ \phi(t)^{\kappa(m-2)} (a^{1/2})^w\partial_{x}^\alpha f  \in   L^2\inner{]0, +\infty[\times \mathbb R_{x,v}^6},$$
 and moreover 
  \begin{multline}\label{malpha}
  	   \sup_{t>0}  \phi(t)^{\kappa(m-2)}\norm{ \partial_{x}^{\alpha} f(t)}_{L^2}+\inner{\int_0^{+\infty} \phi(t)^{2\kappa(m-2)} \norm{ \comi{D_x}^{\frac{s}{1+2s}}\partial_{x}^{\alpha} f(t)}_{L^2}^2 dt}^{1/ 2}\\
	  	   +\inner{\int_0^{+\infty}  \phi(t)^{2\kappa(m-2)} \norm{ (a^{1/2})^w\partial_{x}^{\alpha} f(t)}_{L^2}^2 dt}^{1/ 2} 
	  	\leq \ C_0^{ m-3}\com{ (m-4) !}^{\frac{1+2s}{2s}}.
  \end{multline}
 \end{proposition}

Before proving the above proposition we first state   the interpolation inequality in Sobolev space which is to be used frequently.  Given  three numbers $r_j$ with $r_1<r_2<r_3$ we have
	\begin{equation}\label{interp}
	\forall\ \eps>0,\ \forall\ u\in H^{r_3},	\quad \norm{\comi{D_{x}}^{r_2} u}_{L^2}\leq     \eps \norm{\comi{D_{x}}^{r_3}u}_{L^2}+ \eps^{-\frac{r_2-r_1}{r_3-r_2}}\norm{\comi{D_{x}}^{r_1}u}_{L^2}.
	\end{equation}

\begin{proof}[Proof of Proposition \ref{pro3.2} (The case of $0<t\leq 1$)]  We first consider the case when $t\in]0, 1],$
 and in this part we will prove  the fact that   $t^{\kappa(m-2)} \partial_{x}^\alpha f \in L^\infty\inner{]0, 1];\ L^2}$    and 
 $$t^{\kappa(m-2)} \comi{D_{x}}^{\frac{s}{1+2s}}\partial_{x}^\alpha f, \  t^{\kappa(m-2)} (a^{1/2})^w\partial_{x}^\alpha f  \in   L^2\inner{]0, 1]\times \mathbb R_{x,v}^6} $$    for    any $\alpha$ with $\abs\alpha=m,$  and moreover the norms of these quantities are controlled  by the right-hand side of \eqref{malpha}.

	To do so we define  the regularization $f_\delta $ of $f$ with $0<\delta\ll 1,$  by setting 
\begin{equation}
	\label{fdelta}
	f_\delta=\Lambda_\delta^{-2} f, \quad \Lambda_\delta=\inner{1+\delta\abs{D_x}^2}^{1/2}.
\end{equation} 
Note $\Lambda_\delta$ is just the Fourier multiplier with the symbol $(1+\delta\abs\xi^2)^{1/2}.$  We have $[T,  \Lambda_\delta^{-2}]=0$ for any   operator  $T$  acting only on $v,$   and $\Lambda_\delta^{-2} $ is uniformly bounded in $L^2$ for $\delta$:
\begin{eqnarray*}
	\forall\ u\in L^2,\quad \norm{\Lambda_\delta^{-2} u}_{L^2}\leq \norm{u}_{L^2}.
\end{eqnarray*}
Observe the Fourier multiplier $\Lambda_\delta^{-2}$ is a bounded operator from $L^2(\mathbb R_x^3)$ to $H^2(\mathbb R_x^3)$ with the norm depending on $\delta.$   As a result it follows from the   assumption \eqref{dayu3+}  that
\begin{equation}\label{inass}
\left\{
\begin{aligned}
& t^{\kappa(m-2)-1/2}    \partial_{x_1}^m  f_\delta \in L^\infty\inner{[0, 1];\ L^2},\\
&  t^{\kappa(m-2)}  \comi{D_x}^{\frac{s}{1+2s}}  \partial_{x_1}^m  f_\delta,\   t^{\kappa(m-2)} (a^{1/2})^w\partial_{x_1}^m  f_\delta\in L^2\inner{[0, 1]\times \mathbb R^6}. 
\end{aligned}
\right. 
\end{equation}
 To simply the notation we will use $C$ in the following discussion to denote different suitable constants  which depend only on  $s,\gamma$ and the constant $\epsilon_0$ in \eqref{smacon},   but independent of $m$ and the number $\delta$ in    the Fourier multiplier $\Lambda_\delta^{-2}$,  and moreover denote by  $C_\eps$  different  constants depending on $\eps$ additionally. 
 \medskip

\noindent\underline{\it Step 1 (Upper bound for the trilinear terms). } Recall $\mathcal A  $ is the bounded operator  given in Proposition \ref{prosub} with the properties in \eqref{bd} fulfilled and $\epsilon_0$ is the number in \eqref{smacon}.    Let  $f_\delta$ be the regularization of $f$ given by  \eqref{fdelta}.   In this step we will show that, for any $\eps>0,$  
\begin{equation}\label{uppder}
\begin{aligned}
	&\int_0^1  t^{2 \kappa  (m-2)}\abs{ \inner{P \partial_{x_1}^m f_\delta, \    \partial_{x_1}^m f_\delta}_{L^2} }dt+\int_0^1  t^{2 \kappa  (m-2)}\abs{ \inner{P \partial_{x_1}^m f_\delta, \ \mathcal A   \partial_{x_1}^m f_\delta}_{L^2} }dt\\
	& \leq    \inner{\eps  +\epsilon_0 C}	 \int_0^1 t^{2 \kappa  (m-2) }   \norm{(a^{1/2})^w   \partial_{x_1}^m f _\delta}_{L^2}^2 dt  +   C_\eps   C_0^{2(m-4)}  \com{(m-4)!}^{\frac{1+2s}{s}}.
	\end{aligned}
	\end{equation}
To  confirm this  we use the fact that  $[P, \  \partial_{x_1}^m\Lambda_\delta^{-2}]=0$  and  that $f$ solves the equation $Pf=\Gamma(f,f);$  this gives, recaling $f_\delta=\Lambda_\delta^{-2} f,$
\begin{eqnarray*}
	P \partial_{x_1}^m f_\delta=\Lambda_\delta^{-2} \partial_{x_1}^m \Gamma(f,f)=\Lambda_\delta^{-2} \sum_{j\leq m} {m\choose j} \Gamma(\partial_{x_1}^j f, \ \partial_{x_1}^{m-j}f).
\end{eqnarray*}
Thus
\begin{equation}\label{dees}
 \int_0^1   t^{2 \kappa  (m-2)}\abs{ \inner{P \partial_{x_1}^m f_\delta, \ \mathcal A    \partial_{x_1}^m f_\delta}_{L^2}}\ dt 
	 \leq    S_1+  S_2+S_3+S_4,
\end{equation}
with
\begin{equation*}
\begin{aligned}
& S_1=  \int_0^1  t^{2 \kappa  (m-2)} \abs{\inner{\Lambda_\delta^{-2} \Gamma( f,\ \partial_{x_1}^{m}f), \ \mathcal A    \partial_{x_1}^m f_\delta}_{L^2}}\ dt,\\
  &S_2=\sum_{1\leq j< [{m\over 2}]} {m\choose j} \int_0^1  t^{2 \kappa  (m-2)} \abs{\inner{\Lambda_\delta^{-2} \Gamma(\partial_{x_1}^j f,\ \partial_{x_1}^{m-j}f), \ \mathcal A    \partial_{x_1}^m f_\delta}_{L^2}}\ dt,\\
	& S_3=\sum_{ [{m\over 2}]\leq  j \leq m-1} {m\choose j}\int_0^1   t^{2 \kappa  (m-2)} \abs{\inner{  \Lambda_\delta^{-2}   \Gamma(\partial_{x_1}^j f, \ \partial_{x_1}^{m-j}f),  \ \mathcal A   \partial_{x_1}^m f_\delta}_{L^2}} dt\\
	& S_4= \int_0^1   t^{2 \kappa  (m-2)} \abs{\inner{  \Lambda_\delta^{-2}   \Gamma(\partial_{x_1}^m f, \ f),  \ \mathcal A   \partial_{x_1}^m f_\delta}_{L^2}} dt, 
	 \end{aligned}
\end{equation*}
where $[m/2]$ stands for the largest integer less than or equal to $m/2.$
We first handle $S_1$ and use the fact that  $(1-\delta\Delta_x)\Lambda_\delta^{-2}=1$ to write 
\begin{multline*}
	\Gamma(f,\ \partial_{x_1}^{m}f)=\Gamma( f,\ \partial_{x_1}^{m}(1-\delta\Delta_x)f_\delta)\\
	=(1-\delta\Delta_x) \Gamma(f,\ \partial_{x_1}^{m}f_\delta)+ 2\sum_{k=1}^3\delta \partial_{x_k} \Gamma(\partial_{x_k} f,\ \partial_{x_1}^{m}f_\delta)- \delta  \Gamma(\Delta_x f,\ \partial_{x_1}^{m}f_\delta).
\end{multline*}
Thus
\begin{eqnarray*}
	S_1&\leq &  \int_0^1  t^{2 \kappa  (m-2)} \abs{\inner{  \Gamma( f,\ \partial_{x_1}^{m}f_\delta), \ \mathcal A    \partial_{x_1}^m f_\delta}_{L^2}}\ dt\\
	&&+2\sum_{k=1}^3 \int_0^1  t^{2 \kappa  (m-2)} \abs{\inner{  \Gamma( \partial_{x_k}f,\ \partial_{x_1}^{m}f_\delta), \ \delta\partial_{x_k} \Lambda_\delta^{-2}  \mathcal A    \partial_{x_1}^m f_\delta}_{L^2}}\ dt\\
	&&+  \int_0^1  t^{2 \kappa  (m-2)} \abs{\inner{ \Gamma( \Delta_xf,\ \partial_{x_1}^{m}f_\delta), \  \delta\Lambda_\delta^{-2} \mathcal A    \partial_{x_1}^m f_\delta}_{L^2}}\ dt\\
	&\stackrel{\rm def}{=}& S_{1,1}+S_{1,2}+S_{1,3}.
\end{eqnarray*}
    Using  \eqref{tripleest}  and \eqref{bd}  gives
\begin{equation}
\label{s11}
\begin{aligned}
S_{1,1} & \leq C \int_0^1  t^{2 \kappa  (m-2)}   \sum_{\abs \beta \leq 2}\norm{  \partial_x^\beta f}_{L^2} \norm{(a^{1/2})^w  \partial_{x_1}^{m}f_\delta}_{L^2 }    \norm{(a^{1/2})^w  \mathcal A    \partial_{x_1}^m f_\delta }_{L^2}\ dt\\
& \leq   C\sup_{0\leq t\leq 1}\sum_{\abs\beta \leq 2} \norm{  \partial_x^\beta f(t)}_{L^2}  \int_0^1 t^{2 \kappa  (m-2) }   \norm{(a^{1/2})^w \partial_{x_1}^m f_\delta }_{L^2}^2 dt\\
& \leq   \epsilon_0 C   \int_0^1 t^{2 \kappa  (m-2) }   \norm{(a^{1/2})^w \partial_{x_1}^m f_\delta }_{L^2}^2 dt,
	\end{aligned}
\end{equation}
the last inequality using \eqref{smacon}.  As for $S_{1,2}$ we use \eqref{newesti} and  the fact    that the operators $\Lambda_\delta^{-2}$ and $\delta\partial_{x_j}\partial_{x_k}\Lambda_\delta^{-2}$ are uniformly bounded in $L^2$ with respect to $\delta$ and both commute with $(a^{1/2})^w,$ to compute
\begin{equation}\label{s12}
	\begin{aligned}
S_{1,2} & \leq C \int_0^1  t^{2 \kappa  (m-2)} \Big(  \sum_{\abs \beta \leq 2}\norm{  \partial_x^\beta f}_{L^2}\Big) \norm{(a^{1/2})^w  \partial_{x_1}^{m}f_\delta}_{L^2 } \\
&\qquad\qquad\qquad \times  \Big( \sum_{\abs \beta \leq 2}  \norm{\delta \partial_x^\beta  \Lambda_\delta^{-2}  (a^{1/2})^w    \mathcal A    \partial_{x_1}^m f_\delta }_{L^2}\Big)\ dt\\
& \leq   C\sup_{0\leq t\leq 1}\sum_{\abs\beta \leq 2} \norm{  \partial_x^\beta f(t)}_{L^2}  \int_0^1 t^{2 \kappa  (m-2) }   \norm{(a^{1/2})^w \partial_{x_1}^m f_\delta }_{L^2}^2 dt\\
& \leq   \epsilon_0 C   \int_0^1 t^{2 \kappa  (m-2) }   \norm{(a^{1/2})^w \partial_{x_1}^m f_\delta }_{L^2}^2 dt,
	\end{aligned}
\end{equation}
the last inequality using again \eqref{smacon}.  Similarly we use \eqref{tripleest++++} to get 
\begin{equation*}
	\begin{aligned}
S_{1,3} & \leq C \int_0^1  t^{2 \kappa  (m-2)}  \norm{ \Delta_x f}_{L^2} \norm{(a^{1/2})^w  \partial_{x_1}^{m}f_\delta}_{L^2 }   \Big( \sum_{\abs \beta \leq 2}  \norm{\delta \partial_x^\beta  \Lambda_\delta^{-2}  (a^{1/2})^w    \mathcal A    \partial_{x_1}^m f_\delta }_{L^2}\Big)\ dt\\
& \leq   C\sup_{0\leq t\leq 1}  \norm{ \Delta_x f(t)}_{L^2}  \int_0^1 t^{2 \kappa  (m-2) }   \norm{(a^{1/2})^w \partial_{x_1}^m f_\delta }_{L^2}^2 dt\\
& \leq   \epsilon_0 C   \int_0^1 t^{2 \kappa  (m-2) }   \norm{(a^{1/2})^w \partial_{x_1}^m f_\delta }_{L^2}^2 dt.
	\end{aligned}
\end{equation*}
This along with \eqref{s11} and \eqref{s12} gives
\begin{equation}
\label{s1upp}
S_1\leq 	\epsilon_0 C   \int_0^1 t^{2 \kappa  (m-2) }   \norm{(a^{1/2})^w \partial_{x_1}^m f_\delta }_{L^2}^2 dt.
\end{equation}
Next we treat $S_2$ and use 
\eqref{tripleest} and \eqref{bd}  again to compute
\begin{equation*}
\begin{aligned}
 S_2\leq & C  \sum_{1\leq j< [m/2]} \frac{m !} {j!(m-j)!} \int_0^1  t^{2 \kappa  (m-2)} \bigg[ \Big(\sum_{\abs \beta \leq 2}\norm{ \partial_{x_1}^j\partial_x^\beta f}_{L^2}\Big) \norm{(a^{1/2})^w   \partial_{x_1}^{m-j}f}_{L^2 }\\
&\qquad\qquad\qquad\qquad \qquad\qquad\qquad\qquad   \times \norm{(a^{1/2})^w       \partial_{x_1}^m f_\delta }_{L^2}\bigg]\ dt \\
  \leq &  C \sum_{1\leq j< [m/2]} \frac{m !} {j!(m-k)!}  \inner{ \int_0^1 t^{2 \kappa  (m-2) }   \norm{(a^{1/2})^w    \partial_{x_1}^m f _\delta}_{L^2}^2 dt}^{1/2}  \\
	& \quad \times \sup_{0<t\leq 1}\sum_{\abs\beta \leq 2}t^{\kappa j} \norm{ \partial_{x_1}^{j}\partial_x^\beta f(t)}_{L^2} \inner{\int_0^1  t^{2\kappa  (m-j-2)}  \norm{(a^{1/2
	})^w \partial_{x_1}^{m-j}f}_{L^2 }^2  dt}^{1/2}.	
	\end{aligned}
\end{equation*}
  Moreover  using  the assumptions \eqref{indass} and \eqref{dayu3+} for $1\leq j< [m/2]$,  we compute, for any $\abs\beta\leq2,$
\begin{equation*}
\begin{aligned}
 	 \sup_{0<t\leq 1} t^{\kappa j} \norm{ \partial_{x_1}^{j}\partial_x^\beta f(t)}_{L^2} \leq &\ 
	 \left
\{
\begin{aligned}	  
&\sup_{0<t\leq 1} t^{\kappa (j+\abs\beta-2)} \norm{\partial_{x_1}^{j}\partial_x^\beta f(t)}_{L^2}, \  \textrm{if} \   j +\abs\beta\geq 2,\\
& \sup_{0<t\leq 1}   \norm{\partial_{x_1}^{j}\partial_x^\beta f(t)}_{L^2}, \  \textrm{if} \ j +\abs\beta\leq 1,
\end{aligned}
\right.\\
 \leq &  \   C C_0^{j-1}\com{\inner{ j-1}! }^{\frac{1+2s}{2s}}
\end{aligned}
\end{equation*}
 and     
\begin{eqnarray*}
  \inner{\int_0^1 t^{2 \kappa  (m-j-2 ) }  \norm{(a^{1/2})^w \partial_{x_1}^{m-j}f}_{L^2 }^2  \ dt}^{1/2} 
 \le C_0^{m-j-3}\com{(m-j-4)!} ^{\frac{1+2s}{2s}}.
\end{eqnarray*}
As a result we put these inequalities into the estimate on  $S_2$ to obtain
\begin{multline*}
  S_2 \leq  C  	\inner{ \int_0^1 t^{2 \kappa  (m-2) }   \norm{(a^{1/2})^w    \partial_{x_1}^m f _\delta}_{L^2}^2 dt}^{1/2}\\
  \times   \sum_{1\leq j < [m/2]} \frac{m !} {j!(m-j)!}  C_0^{j-1}\com{\inner{ j-1}! }^{\frac{1+2s}{2s}}  \Big(  C_0^{m-j-3}\com{(m-j-4)!} ^{\frac{1+2s}{2s}}   \Big),
\end{multline*}
and direct computation shows that
\begin{eqnarray*}
	&&\sum_{1\leq j < [m/2]} \frac{m !} {j!(m-j)!}  C_0^{j-1}\com{\inner{ j-1}! }^{\frac{1+2s}{2s}}  \Big(  C_0^{m-j-3}\com{(m-j-4)!} ^{\frac{1+2s}{2s}}  \Big)\\
	&\leq &C  C_0^{m-4}\sum_{1\leq j < [m/2]} \frac{m !} {j (m-j)^{4}}   \com{\inner{ j-1}! }^{\frac{1}{2s}}  \com{(m-j-4)!} ^{\frac{1}{2s}}\\
	&\leq & C  C_0^{m-4}\sum_{1\leq j < [m/2]} \frac{m !} {j (m-j)^{4}}      \com{(m-5)!} ^{\frac{1}{2s}}\\
	& \leq & C  C_0^{m-4}  \com{(m-4)!} ^{\frac{1+2s}{2s}} \sum_{1\leq j < [m/2]} \frac{1} { j  m^{\frac{1}{2s}}}   \leq   C  C_0^{m-4}  \com{(m-4)!} ^{\frac{1+2s}{2s}}, 
	\end{eqnarray*}
	the last inequality holding because  $$\frac{1}{m^{\frac{1}{2s}}}\sum_{1\leq j < [m/2]} \frac{1} { j  } \leq C_s$$ with $C_s$ a constant depending only on $s. $
Thus we combine  the above inequalities to  obtain,  for any $\eps>0,$ 
	\begin{equation}
	\label{s2upp}
		S_2\leq \eps  	 \int_0^1 t^{2 \kappa  (m-2) }   \norm{(a^{1/2})^w    \partial_{x_1}^m f _\delta}_{L^2}^2 dt +   C_\eps   C_0^{2(m-4)}  \com{(m-4)!}^{\frac{1+2s}{s}}.
	\end{equation}
	 The treatment of $S_3$ is similar as that of $S_2$,  using \eqref{tripleest+} here instead of \eqref{tripleest}.  Meanwhile following the argument for handling $S_1$ will yield  the upper bound of $S_4.$       For brevity we omit the details  and conclude that 
\begin{eqnarray*}
		S_3  \leq  \eps  	 \int_0^1 t^{2 \kappa  (m-2) }   \norm{(a^{1/2})^w    \partial_{x_1}^m f _\delta}_{L^2}^2 dt +   C_\eps   C_0^{2(m-4)}  \com{(m-4)!}^{\frac{1+2s}{s}}
				\end{eqnarray*}
	and
	\begin{eqnarray*}
		S_4\leq 	\epsilon_0 C   \int_0^1 t^{2 \kappa  (m-2) }   \norm{(a^{1/2})^w \partial_{x_1}^m f_\delta }_{L^2}^2 dt.
	\end{eqnarray*} 
This along with  the estimates \eqref{s1upp}-\eqref{s2upp} on $S_1$ and $S_2$  as well as  \eqref{dees} yields the desired upper bound for the second term on the left-hand side of  \eqref{uppder}.  Meanwhile the  first term can be handled in the same way with simpler argument.  Then we have proven \eqref{uppder}.
 
\medskip 
\noindent\underline{\it Step 2.}  In this step we will derive the desired estimate \eqref{dayu3+} for short time $0<t\leq 1,$  that is,   for any multi-index $\alpha$ with $\abs\alpha=m,$ we have 
\begin{multline}\label{esfor01}
		\sup_{0<t\leq 1}t^{\kappa(m-2)}\norm{ \partial_{x}^\alpha f(t)}_{L^2} +\inner{\int_0^1 t^{2\kappa(m-2)} \norm{\comi{D_{x}}^{\frac{s}{1+2s}}  \partial_{x}^\alpha f(t)}_{L^2}^2dt}^{1/2}\\
	+\inner{\int_0^1 t^{2\kappa(m-2)} \norm{(a^{1/2})^w \partial_{x}^\alpha f(t)}_{L^2}^2dt}^{1/2} \leq  \frac{1}{2}C_{0}^{m-3}[(m-4)!]^{\frac{1+2s}{2s}}.
	\end{multline}
To do so,  by \eqref{inass}  we  can apply the subelliptic estimate \eqref{subest} with $u= \partial_{x_1}^m f_\delta$ and $r=\kappa(m-2)$;  this gives  for any $0<t\leq 1,$
\begin{multline}\label{enestofr}
   t^{2\kappa\inner{  m-2}} \norm{ \partial_{x_1}^m  f_\delta}_{L^2}^2+\int_0^1  t^{2\kappa\inner{  m-2}} \norm{ \comi{D_x}^{\frac{s}{1+2s}} \partial_{x_1}^m  f_\delta}_{L^2}^2dt\\
  +\int_0^1  t^{2\kappa\inner{  m-2}} \norm{  (a^{1/2})^w  \partial_{x_1}^m f_\delta}_{L^2}^2 \leq  \mathcal M,
	\end{multline}
	where
	\begin{multline*}
	\mathcal M=  C\int_0^1  t^{2\kappa\inner{  m-2}} \abs{\inner{P   \partial_{x_1}^m f_\delta, \   \partial_{x_1}^m f_\delta}_{L^2}} dt\\  +C \int_0^1  t^{2\kappa\inner{  m-2}} \abs{\inner{P    \partial_{x_1}^m f_\delta, \ \mathcal A      \partial_{x_1}^m f_\delta}_{L^2}} +C m \int_0^1  t^{2\kappa\inner{  m-2}-1} \left\|  \partial_{x_1}^m f_\delta\right\|_{L^2}^2.
\end{multline*}
 As for the last term above we use the interpolation inequality \eqref{interp} to get,  for any $\eps>0,$ 
\begin{equation*}
\begin{aligned}
	 & m t^{2\kappa\inner{  m-2}-1}\norm{  \partial_{x_1}^m f_\delta}_{L^2}^2\\
	 \leq & \ \eps t^{2\kappa\inner{  m-2}}\norm{ \comi{D_{x}}^{\frac{s}{1+2s}} \partial_{x_1}^m f_\delta}_{L^2}^2 +C_\eps m^{\frac{1+2s}{s}}t^{2\kappa (m-3)}\norm{ \comi{D_x}^{\frac{s}{1+2s}}  \partial_{x_1}^{m-1} f_\delta}_{L^2}^2\\
	 \leq & \ \eps t^{2\kappa\inner{  m-2}}\norm{ \comi{D_{x}}^{\frac{s}{1+2s}} \partial_{x_1}^m f_\delta}_{L^2}^2 +C_\eps m^{\frac{1+2s}{s}}t^{2\kappa (m-3)}\norm{ \comi{D_x}^{\frac{s}{1+2s}}  \partial_{x_1}^{m-1} f}_{L^2}^2,
	 \end{aligned}
\end{equation*}
recalling    $\kappa=\frac{1+2s}{2s}.$
As a result, we use the assumption \eqref{dayu3+} to compute   
\begin{equation}\label{keyes}
\begin{aligned}
	&C m \int_0^1  t^{2\kappa\inner{  m-2}-1} \left\|   \partial_{x_1}^m f_\delta\right\|_{L^2}^2\\
	 \leq &   \eps \int_0^1  t^{2\kappa\inner{  m-2}}\norm{  \comi{D_{x}}^{\frac{s}{1+2s}}  \partial_{x_1}^m f_\delta}_{L^2}^2 dt +C_\eps m^{\frac{1+2s}{s}}\int_0^1 t^{2\kappa (m-3)}\norm{ \comi{D_x}^{\frac{s}{1+2s}}  \partial_{x_1}^{m-1} f}_{L^2}^2dt\\
	\leq &  \eps \int_0^1  t^{2\kappa\inner{  m-2}}\norm{  \comi{D_{x}}^{\frac{s}{1+2s}}  \partial_{x_1}^m f_\delta}_{L^2}^2 dt  +C_\eps  C_0^{2(m-4)}m^{\frac{1+2s}{s}}  \com{ (m-5) !}^{\frac{1+2s}{s}}\\
	\leq &  \eps \int_0^1  t^{2\kappa\inner{  m-2}}\norm{  \comi{D_{x}}^{\frac{s}{1+2s}} \partial_{x_1}^m f_\delta}_{L^2}^2 dt+C_\eps C_0^{2(m-4)}     \com{ (m-4) !}^{\frac{1+2s}{s}}.
\end{aligned}
\end{equation}
  Combining the above inequality and \eqref{uppder}, we get the upper bound of the term $\mathcal M$ on the right-hand side of \eqref{enestofr}; that is,
\begin{multline*}
	\mathcal M \leq   
  \inner{\eps  +\epsilon_0 C	} \int_0^1 t^{2 \kappa  (m-2) }   \norm{(a^{1/2})^w    \partial_{x_1}^m f _\delta}_{L^2}^2 dt\\
+  \eps \int_0^1  t^{2\kappa\inner{  m-2}}\norm{  \comi{D_{x}}^{\frac{s}{1+2s}}  \partial_{x_1}^m f_\delta}_{L^2}^2 dt +C_\eps C_0^{2(m-4)}     \com{ (m-4) !}^{\frac{1+2s}{s}}.  
\end{multline*}
  Suppose   $\epsilon_0$ is  small enough and  let  $\eps$ be small  as well    such that the first two terms on the right-hand side of the above inequality  can be    absorbed  by  the left ones  in \eqref{enestofr}.  Thus we conclude for any $0<t\leq 1$,  
\begin{multline*}
	  t^{2\kappa\inner{  m-2}} \norm{   \partial_{x_1}^m  f_\delta}_{L^2}^2+ \int_0^1  t^{2\kappa\inner{  m-2}} \norm{ \comi{D_x}^{\frac{s}{1+2s}}  \partial_{x_1}^m  f_\delta}_{L^2}^2dt\\
    + \int_0^1  t^{2\kappa\inner{  m-2}} \norm{  (a^{1/2})^w    \partial_{x_1}^m f_\delta}_{L^2}^2 \leq C  C_0^{2(m-4)}     \com{ (m-4) !}^{\frac{1+2s}{s}}.
		\end{multline*}
	 Since the constants $C$ and $C_0$ above are independent of $\delta,$  then letting $\delta\rightarrow 0$  implies that  $$ t^{\kappa\inner{  m-2}}    \partial_{x_1}^m  f\in   L^\infty\inner{]0,1];\ L^2} $$ 
	 and
	 \begin{eqnarray*}
	 	t^{\kappa\inner{  m-2}}\comi{D_x}^{\frac{s}{1+2s}}   \partial_{x_1}^m  f, \  t^{\kappa\inner{  m-2}}  (a^{1/2})^w \partial_{x_1}^m  f \in L^2\inner{]0,1]\times \mathbb R^6},
	 \end{eqnarray*}
and moreover  that 
	\begin{multline}\label{f1}
	  \sup_{0<t\leq 1}t^{\kappa\inner{  m-2}} \norm{   \partial_{x_1}^m  f}_{L^2}+\inner{\int_0^1  t^{2\kappa\inner{  m-2}} \norm{ \comi{D_x}^{\frac{s}{1+2s}}  \partial_{x_1}^m  f}_{L^2}^2dt}^{1/2}\\
    +\inner{ \int_0^1  t^{2\kappa\inner{  m-2}} \norm{  (a^{1/2})^w    \partial_{x_1}^m f}_{L^2}^2}^{1/2} \leq C  C_0^{m-4}     \com{ (m-4) !}^{\frac{1+2s}{2s}}.
		\end{multline} 
		The above estimate obviously  holds with $\partial_{x_1}^m$ replaced by $\partial_{x_j}^m, j=2, 3.$  Then using the fact that $\norm{\partial_{x}^{\alpha}  f}_{L^2 }^2\leq \norm{\partial_{x_1}^{m}  f}_{L^2 }^2+\norm{\partial_{x_2}^{m}  f}_{L^2 }^2+\norm{\partial_{x_3}^{m}  f}_{L^2 }^2$ for any $\abs\alpha=m,$  we conclude   for any $\alpha$ with $\abs\alpha=m,$ 
 	\begin{multline*}
	\sup_{0<t\leq 1}  t^{\kappa\inner{  m-2}} \norm{   \partial_{x}^\alpha  f}_{L^2}+\inner{\int_0^1  t^{2\kappa\inner{  m-2}} \norm{ \comi{D_x}^{\frac{s}{1+2s}}  \partial_{x}^\alpha  f}_{L^2}^2dt}^{1/2}\\
    +\inner{\int_0^1  t^{2\kappa\inner{  m-2}} \norm{  (a^{1/2})^w    \partial_{x}^\alpha f}_{L^2}^2}^{1/2} \leq C C_0^{m-4}     \com{ (m-4) !}^{\frac{1+2s}{2s}}.
		\end{multline*} 
Then 	the desired estimate \eqref{esfor01} follows  if we take $C_0\geq 2 C$ with $C$ the constant in the above inequality.    
\end{proof}

\begin{proof}
	[Completeness of the proof of Proposition \ref{pro3.2} (the case of $t\geq 1$)]   
	  It remains to prove the validity of \eqref{malpha}  in Proposition \ref{pro3.2}   when $t>1.$    The proof is quite similar as in the case of $0<t\leq 1,$ and the argument here will be simpler since this part is just the propagation property of Gevrey regularity.   
	Indeed,  we apply the   estimate \eqref{1infty} for $u=  \partial_{x_1}^m f_\delta$;  this gives, for any $t\geq 1,$  
	\begin{equation}\label{enestofr++}
	\begin{aligned}
  &   \norm{   \partial_{x_1}^m  f_\delta(t)}_{L^2}^2+ \int_1^{+\infty}   \norm{  \comi{D_x}^{\frac{s}{1+2s}}    \partial_{x_1}^m f_\delta}_{L^2}^2+ \int_1^{+\infty}   \norm{  (a^{1/2})^w    \partial_{x_1}^m f_\delta}_{L^2}^2\\
  \leq & \  \norm{   \partial_{x_1}^m  f_\delta(1)}_{L^2}^2 +   C\int_1^{+\infty}   \abs{\inner{P   \partial_{x_1}^m f_\delta, \   \partial_{x_1}^m f_\delta}_{L^2}} dt\\
  & +   C\int_1^{+\infty}   \abs{\inner{P   \partial_{x_1}^m f_\delta, \  \mathcal A \partial_{x_1}^m f_\delta}_{L^2}} dt 
 + C\int_1^{+\infty}   \norm{ \partial_{x_1}^m f_\delta}_{L^2}^2 dt.
 \end{aligned}
\end{equation}
As for the first term on the right-hand side, we have  obtained  in \eqref{f1} its upper bound:
\begin{eqnarray*}
	\norm{   \partial_{x_1}^m  f_\delta(1)}_{L^2}^2\leq \norm{   \partial_{x_1}^m  f(1)}_{L^2}^2\leq C C_0^{2(m-4)}  \com{(m-4)!}^{\frac{1+2s}{s}}.
\end{eqnarray*}
Repeating the argument for proving \eqref{uppder}, we see the second and third terms on the right-hand side of \eqref{enestofr++} are bounded from above by 
	\begin{equation*} 
\begin{aligned}
    \inner{\eps  +\epsilon_0 C}	 \int_1^{+\infty}    \norm{(a^{1/2})^w   \partial_{x_1}^m f _\delta}_{L^2}^2 dt  +   C_\eps   C_0^{2(m-4)}  \com{(m-4)!}^{\frac{1+2s}{s}},
	\end{aligned}
	\end{equation*}
	with $\eps$ arbitrarily small. 
	As for the last term in \eqref{enestofr++}, we use  interpolation equality \eqref{interp} and then the assumption \eqref{dayu3+}   to obtain
	\begin{equation*}
\begin{aligned}
	  \int_1^{+\infty}\norm{  \partial_{x_1}^m f_\delta}_{L^2}^2\ dt 
	 \leq & \ \eps \int_1^{+\infty} \norm{ \comi{D_{x}}^{\frac{s}{1+2s}} \partial_{x_1}^m f_\delta}_{L^2}^2 + C_\eps \int_1^{+\infty} \norm{ \comi{D_x}^{\frac{s}{1+2s}}  \partial_{x_1}^{m-1} f_\delta}_{L^2}^2 \\
	 \leq & \ \eps \int_0^1\norm{ \comi{D_{x}}^{\frac{s}{1+2s}} \partial_{x_1}^m f_\delta}_{L^2}^2 \ dt+C_\eps C_0^{2(m-4)}  \com{(m-5)!}^{\frac{1+2s}{s}}.
	 \end{aligned}
\end{equation*}
Finally  supposing  $\epsilon_0$ is small enough and choosing $\eps$ small as well,    we combine  the above inequalities   to  get, for any $t>1,$  
	\begin{eqnarray*}
&& \norm{   \partial_{x_1}^m  f_\delta(t)}_{L^2}^2+\int_1^{+\infty}   \norm{ \comi{D_x}^{\frac{s}{1+2s}}  \partial_{x_1}^m  f_\delta}_{L^2}^2dt +\int_1^{+\infty}  \norm{  (a^{1/2})^w    \partial_{x_1}^m f_\delta}_{L^2}^2\\
		&\leq & C  C_0^{2(m-4)}  \com{(m-4)!} ^{\frac{1+2s}{s}}.
		\end{eqnarray*}	
The remaining argument is just the same as that in the previous case of $0<t\leq 1$, so we omit it here and conclude that,  for any $\abs\alpha=m,$
\begin{multline*}
  \sup_{t\geq 1} \norm{   \partial_{x}^\alpha  f(t)}_{L^2}+\Big(\int_1^{+\infty}   \norm{ \comi{D_x}^{\frac{s}{1+2s}}  \partial_{x_1}^m  f_\delta}_{L^2}^2dt\Big)^{1/2} +\Big(\int_1^{+\infty}  \norm{  (a^{1/2})^w    \partial_{x_1}^m f_\delta}_{L^2}^2\Big)^{1/2}\\
		\leq  \frac{1}{2}  C_0^{m-3}  \com{(m-4)!} ^{\frac{1+2s}{2s}}.
\end{multline*}
 This along with \eqref{esfor01} yields \eqref{malpha} as desired.   The proof of  Proposition \ref{pro3.2} is  completed. 
\end{proof}

 \begin{proposition}\label{pro3.2+} 
 Denote   $\kappa=(1+2s)/2s.$  Assume that the cross-section satisfies \eqref{kern} and \eqref{angu}  with $0 < s < 1$  and $\gamma \geq 0.$
 	Let  $f\in L^\infty\inner{[0,+\infty];~H^2}$  be any solution to \eqref{3} satisfying \eqref{smacon}.        Then  there exists a constant $C_*,$ depending only on $s,\gamma$ and the constant $\epsilon_0$ in \eqref{smacon},  such that 
  for any multi-index $\beta$ with  $2\leq \abs\beta\leq 4$ we have
   \begin{multline}\label{indass+}
   \sup_{t>0}  \phi(t)^{\kappa(\abs\beta-2)}\norm{ \partial_{x}^{\beta} f(t)}_{L^2}+\inner{\int_0^{+\infty} \phi(t)^{2\kappa(\abs\beta-2)} \norm{ \comi{D_x}^{\frac{s}{1+2s}}\partial_{x}^{\beta} f(t)}_{L^2}^2 dt}^{1/ 2}\\
	  	   +\inner{\int_0^{+\infty}  \phi(t)^{2\kappa(\abs\beta-2)} \norm{ (a^{1/2})^w\partial_{x}^{\beta} f(t)}_{L^2}^2 dt}^{1\over 2}  \leq  
	  	  C_*.	
   \end{multline}
 \end{proposition}
 
 \begin{proof}
 	The proof is quite similar as that of  Proposition \ref{pro3.2}.  In fact using a similar subelliptic estimate as \eqref{1infty} and  repeating the procedure for proving Proposition \ref{pro3.2}  we have
 	\begin{eqnarray*}
 		 \sum_{\abs\beta=2}\inner{\int_0^{+\infty}  \norm{ \comi{D_x}^{\frac{s}{1+2s}}\partial_{x}^{\beta} f(t)}_{L^2}^2 dt}^{1/ 2}\leq C,
 	\end{eqnarray*}
 	for some some constant  $C$  depending only on $s,\gamma$ and the constant $\epsilon_0.$   This, along with the assumption \eqref{smacon}, yields the validity of \eqref{indass+} for $\abs\beta=2.$ Furthermore repeating again the argument for proving Proposition \ref{pro3.2} we can verify directly that \eqref{indass+} holds for  $\abs\beta=3$ and then for $ \abs\beta= 4.$ 	Since   the argument involved here is direct and simpler than the one in Proposition \ref{pro3.2},  we omit it for brevity.     
 \end{proof}

Now we can prove the main result on the Gevrey regularization in spatial variable. 

\begin{proof}
	[Proof of  Theorem \ref{thm3.1}] By Proposition \ref{pro3.2+} we see the assumption \eqref{indass}  in Proposition \ref{pro3.2}  holds and moreover the induction assumption \eqref{dayu3+}  holds for any $\beta$ with $\abs\beta= 4$ provided $C_0\geq C_*.$
	   This along with  Proposition \ref{pro3.2} enables to 
	   use induction  to obtain  that   for any $\alpha\in\mathbb Z_+^3$ with $\abs\alpha\geq 4,$   we have  
	   \begin{multline}\label{onspa}
	 	  	\sup_{t>0}  \phi(t)^{\kappa(\abs\alpha -2)}\norm{ \partial_{x}^{\alpha } f(t)}_{L^2}+\inner{\int_0^{+\infty} \phi(t)^{2\kappa(\abs\alpha -2)} \norm{ \comi{D_x}^{\frac{s}{1+2s}}  \partial_{x}^{\alpha } f(t)}_{L^2}^2 dt}^{1/2}\\
	  	+\inner{\int_0^{+\infty}  \phi(t)^{2\kappa(\abs\alpha -2)} \norm{ (a^{1/2})^w\partial_{x}^{\alpha } f(t)}_{L^2}^2 dt}^{1/2} 
	  	\leq \ C_0^{ \abs\alpha -3}\com{ (\abs\alpha -4) !}^{\frac{1+2s}{2s}}.
	\end{multline}
	As a result, for any $t>0$  and   any $\abs\alpha \geq 0$ we have
	\begin{equation*}
		\phi(t)^{\kappa \abs\alpha  }\norm{ \partial_{x}^{\alpha } f(t)}_{L^2}\leq
		\left\{
		\begin{aligned} 
		& \norm{ \partial_{x}^{\alpha } f(t)}_{L^2},\quad \textrm{if}\ \abs\alpha \leq 2\\
		&\phi(t)^{\kappa(\abs\alpha -2)}\norm{ \partial_{x}^{\alpha } f}_{L^2},\  \textrm{if}\ \abs\alpha  \geq 3
		\end{aligned}
\right.  
\leq C_0^{ \abs\alpha +1}\inner{ \abs\alpha  !}^{\frac{1+2s}{2s}},
	\end{equation*}
	completing the proof of Theorem   \ref{thm3.1}.  
\end{proof}

\section{Gevrey regularization  in velocity variable}

As in the previous section,  the Gevrey regularization for $v$ variable is just an immediate consequence of the following proposition.

\begin{proposition}\label{progevofv}
 Denote  $\kappa=(1+2s)/2s$ and  let $m\geq 5$ be an arbitrarily given integer.   Let  $f\in L^\infty\inner{[0,+\infty[;  H^2}$ be any solution  to the Cauchy problem  \eqref{3} satisfying the condition \eqref{smacon}.     Suppose additionally that there exists a positive constant $\tilde C_*\geq 1$,  depending only on $s,\gamma$ and the constant $\epsilon_0$ in \eqref{smacon} such that for any multi-index $\beta$ with  $2\leq \abs\beta\leq 3$ we have
   \begin{equation*} 
   \sup_{t>0}  \phi(t)^{\kappa(\abs\beta-2)}\norm{ \partial_{v}^{\beta} f(t)}_{L^2} +\inner{\int_0^{+\infty}  \phi(t)^{2\kappa(\abs\beta-2)} \norm{ (a^{1/2})^w\partial_{v}^{\beta} f(t)}_{L^2}^2 dt}^{1\over 2}  \leq  
	  	  \tilde C_*.	
   \end{equation*}
 Then we can find a constant  $\tilde C_0\geq \max\{\tilde C_*, C_0^2\} $ with $C_0$   given in Theorem \ref{thm3.1},    depending only on $s,\gamma$ and the constant $\epsilon_0$ in \eqref{smacon}  but independent of $m,$   such that 
  if for any multi-index  $ \beta$ with $4\leq \abs\beta\leq m-1$ we have
 \begin{equation}\label{asonv}
 \begin{aligned}
	  &	\sup_{t>0}\phi(t)^{\kappa( \abs\beta-2)}\norm{  \partial_v^\beta f(t)}_{L^2}+\inner{\int_0^\infty \phi(t)^{2\kappa(\abs\beta-2)} \norm{ (a^{1/2})^w\partial_v^\beta f(t)}_{L^2}^2 dt}^{1/ 2}\\
	  \leq  & \tilde C_0^{  \abs\beta-3}\com{ ( \abs\beta-4) !}^{\frac{1+2s}{2s}},
	  \end{aligned}
	\end{equation}
    then the above estimate \eqref{asonv} still holds for any $\beta$ with $\abs\beta=m.$ 
\end{proposition}

\begin{proof} Since the procedure of the proof  is  quite similar as in Proposition \ref{pro3.2},  we only  give a sketch here and will emphasize on the difference.  In the following argument we use the notation that
  \begin{eqnarray*}
  	f_{m,\delta}=  \big(1+\delta\abs{D_{v}}^2\big)^{-1} \partial_{v_1}^m  f.
  \end{eqnarray*}
  Apply Proposition \ref{prpellp} for   $u= 
  f_{m,\delta}$ and $r=\kappa(m-2)$;  this gives for any $0<t\leq 1,$
\begin{equation}\label{tildereterm}
\begin{aligned}
&  t^{2\kappa\inner{  m-2}} \norm{  f_{m,\delta}(t)}_{L^2}^2+ 
   \int_0^1  t^{2\kappa\inner{  m-2}} \norm{  (a^{1/2})^w  f_{m,\delta}}_{L^2}^2\\
\leq &   C\int_0^1  t^{2\kappa\inner{  m-2}} \abs{\inner{P    f_{m,\delta}, \   f_{m,\delta}}_{L^2}} dt +C m \int_0^1  t^{2\kappa\inner{  m-2}-1} \left\|    f_{m,\delta}\right\|_{L^2}^2. 
   \end{aligned}
	\end{equation}
For the last term on the right side of \eqref{tildereterm} we can repeat the argument for proving \eqref{keyes}, to obtain, for any $\eps, \tilde \eps>0, $
\begin{equation}\label{lae}
\begin{aligned}
&C m \int_0^1  t^{2\kappa\inner{  m-2}-1} \left\|     f_{m,\delta}\right\|_{L^2}^2\\
\leq & 	\tilde \eps  \int_0^1  t^{2\kappa\inner{m-2}} \norm{\comi{D_{v}}^{\frac{s}{1+2s}}   f_{m,\delta}}_{L^2}^2+C_{\tilde\eps} m^{\frac{1+2s}{s}} \int_0^1  t^{2\kappa\inner{  m-3}} \norm{ \comi{D_{v}}^{\frac{s}{1+2s}}f_{m-1,\delta}}_{L^2}^2\\
	\leq  & \eps \int_0^1  t^{2\kappa\inner{  m-2}}\norm{  (a^{1/2})^w    f_{m,\delta}}_{L^2}^2 dt+C_\eps m^{\frac{1+2s}{s}}\int_0^1  t^{2\kappa\inner{  m-3}}\norm{  (a^{1/2})^w     \partial_{v_1}^{m-1}f}_{L^2}^2 dt\\
	\leq & \eps  \int_0^1  t^{2\kappa\inner{  m-2}}\norm{  (a^{1/2})^w   f_{m,\delta}}_{L^2}^2 dt+C_\eps \tilde C_0^{2(m-4)}     \com{ (m-4) !}^{\frac{1+2s}{s}},
	\end{aligned}
\end{equation}
the second equality holding because we may write
\begin{eqnarray*}
 \comi{D_{v}}^{\frac{s}{1+2s}} =\underbrace{ \comi{D_{v}}^{\frac{s}{1+2s}} \big[(a^{1/2})^w \big]^{-1}}_{\textrm{bounded operator}}(a^{1/2})^w
\end{eqnarray*}
due to the conclusions $(i)$ and $(iii)$ in Proposition \ref{estaa}, and the last inequality using the assumption \eqref{asonv}.

Next we treat the first term on the right side of \eqref{tildereterm}, the main different part from the previous section.   Observe 
\begin{multline*}
	 P   f_{m,\delta} =   \big(1+\delta\abs{D_{v}}^2\big)^{-1} \partial_{v_1}^m P f+  \big(1+\delta\abs{D_{v}}^2\big)^{-1} \big[ P,   \  \partial_{v_1}^m\big]f\\
	 +\big[ P,   \ \big(1+\delta\abs{D_{v}}^2\big)^{-1} \big]\partial_{v_1}^mf
\end{multline*}
and furthermore $P f=\Gamma(f,f).$
Then 
\begin{equation*}
	\int_0^1  t^{2\kappa\inner{  m-2}} \abs{\inner{P  f_{m,\delta}, \   f_{m,\delta}}_{L^2}} dt\leq 
\sum_{\ell =1}^3 \mathcal J_\ell,
\end{equation*}
with
\begin{eqnarray*}
		\mathcal J_{1}&=&\int_0^1  t^{2\kappa\inner{  m-2}}\Big|\Big(   \big(1+\delta\abs{D_{v}}^2\big)^{-1} \partial_{v_1}^m \Gamma (f,f),\ ~    f_{m,\delta}\Big)_{L^2}\Big|\ dt,\\
		\mathcal J_{2}&=&\int_0^1  t^{2\kappa\inner{  m-2}}\Big| \Big(   \big(1+\delta\abs{D_{v}}^2\big)^{-1}\com{ P,   \  \partial_{v_1}^m}f,  ~   f_{m,\delta}\Big)_{L^2}\Big|\ dt,\\
		\mathcal J_{3}&=&	\int_0^1  t^{2\kappa\inner{  m-2}}\Big|\Big(\big[ P,   \     \big(1+\delta\abs{D_{v}}^2\big)^{-1} \big]\partial_{v_1}^m f,\ ~   f_{m,\delta}\Big)_{L^2}\Big|\ dt.
\end{eqnarray*}

  \medskip
  
   \noindent\underline{\it Estimate on  $\mathcal J_1$}.  The $\mathcal J_1$ can be handled in the same way as the terms $S_j$ defined in \eqref{dees}.  Here we have to handle the commutator between $(a^{1/2})^w$ and $\big(1+\delta\abs{D_{v}}^2\big)^{-1}$ and there is no additional difficulty since 
   \begin{eqnarray*}
   	\big[(a^{1/2})^w, \ \big(1+\delta\abs{D_{v}}^2\big)^{-1}\big]=\underbrace{	\big[(a^{1/2})^w, \ \big(1+\delta\abs{D_{v}}^2\big)^{-1}\big]\big((a^{1/2})^w\big)^{-1}}_{\textrm{uniformly bounded w.r.t. } \delta}(a^{1/2})^w.
   \end{eqnarray*}
   Moreover Leibniz formula also holds  in the form 
   that
   \begin{equation}\label{sumtrilinear}
   	\partial_{v_1}^m \Gamma (f,f)=\sum_{0\leq j\leq m}\sum_{0\leq k\leq j} {m\choose j}{j\choose k} \mathcal T(\partial_{v_1}^{m-j} f, \ \partial_{v_1}^{j-k} f,  \ \partial_{v_1}^{k} \mu^{1/2})
   \end{equation}
   with 
   \begin{equation}\label{matht}
   	  \mathcal T (g,h, \omega)\stackrel{\rm def}{ =}   \iint B(v-v_*,\sigma) \omega_* \inner{g_*'h'-g_*h} dv_*d\sigma.
   \end{equation}
   Note that $\Gamma(g,h)=\mathcal T(g,h,\mu^{1/2})$ and   a  constant $L>0$ exists such that
   \begin{equation}\label{anfordis}
\forall~k\geq 0,\quad  \big| 	 \partial_{v_1}^{k} \mu^{1/2}\big| \leq L^{k+1}k!\mu^{1/4}.
   \end{equation}
Thus the terms in the summation of \eqref{sumtrilinear}  enjoy the same upper bounds  as in \eqref{tripleest}-\eqref{tripleest++++} and \eqref{newesti}, so we can follow the argument for handling $\partial_{x_1}^m\Gamma(f,\ f)$  in the previous section  to conclude that  
   \begin{eqnarray*}
 	\mathcal J_{1}\leq    \inner{\eps  +\epsilon_0 C}	 \int_0^1 t^{2 \kappa  (m-2) }   \norm{(a^{1/2})^w    f _{m,\delta}}_{L^2}^2 dt
	 +   C_\eps   \tilde C_0^{2(m-4)}  \com{(m-4)!}^{\frac{1+2s}{s}}.
 \end{eqnarray*}

\medskip
  \noindent\underline{\it Upper bound of $\mathcal J_2$}.  Note that 
  \begin{eqnarray*}
  	\com{ P,   \  \partial_{v_1}^m}=-m\partial_{x_1}\partial_{v_1}^{m-1}+\com{a^w+\mathcal R, \ \partial_{v_1}^m}
  \end{eqnarray*}
  with $a^w$ and $\mathcal R$ given in Proposition \ref{estaa}.
  Thus
 \begin{eqnarray*}
 	\mathcal J_{2} &\leq&   	m\int_0^1  t^{2\kappa\inner{  m-2}}\Big|\Big(  \big(1+\delta\abs{D_{v}}^2\big)^{-1}  \partial_{x_1}\partial_{v_1}^{m-1} f,\ ~   f_{m,\delta}\Big)_{L^2}\Big|\ dt\\
 	&&+	\int_0^1  t^{2\kappa\inner{  m-2}}\Big|\Big( \big(1+\delta\abs{D_{v}}^2\big)^{-1}  \big[ a^w+\mathcal R,   \     \partial_{v_1}^m \big] f,\ ~   f_{m,\delta}\Big)_{L^2}\Big|\ dt\\
 	&\stackrel{\rm def}{=}&\mathcal J_{2,1}+\mathcal J_{2,2}.
 \end{eqnarray*}
   We first estimate $\mathcal J_{2,1}$ and use the fact that 
   \begin{eqnarray*}
   	\norm{ \big(1+\delta\abs{D_{v}}^2\big)^{-1}  \partial_{x_1}\partial_{v_1}^{m-1} f}_{L^2}\leq  C 	\norm{    \partial_{x_1}^m  f}_{L^2}+	C \norm{ \big(1+\delta\abs{D_{v}}^2\big)^{-1}   \partial_{v_1}^{m} f}_{L^2},
   \end{eqnarray*}
   to get
   \begin{equation*}
   \begin{aligned}
   	\mathcal J_{2,1}\leq& C m  \int_0^1  t^{2\kappa\inner{  m-2}}\norm{    \partial_{x_1}^m  f}_{L^2}^2\ dt +Cm  \int_0^1  t^{2\kappa\inner{  m-2}}\norm{   f_{m,\delta}}_{L^2}^2\ dt\\
  \leq& 	\eps  \int_0^1  t^{2\kappa\inner{  m-2}}\norm{  \comi{D_x}^{\frac{s}{1+2s}}  \partial_x^m f}_{L^2}^2 dt+	\eps  \int_0^1  t^{2\kappa\inner{  m-2}}\norm{  (a^{1/2})^w   f_{m,\delta}}_{L^2}^2 dt\\
  &+C_\eps C_0^{2(m-4)}     \com{ (m-4) !}^{\frac{1+2s}{s}}+C_\eps\tilde C_0^{2(m-4)}     \com{ (m-4) !}^{\frac{1+2s}{s}}\\
  \leq& 	\eps C_0^{2(m-3)}     \com{ (m-4) !}^{\frac{1+2s}{s}}+	\eps  \int_0^1  t^{2\kappa\inner{  m-2}}\norm{  (a^{1/2})^w   f_{m,\delta}}_{L^2}^2 dt\\
  &+C_\eps C_0^{2(m-4)}     \com{ (m-4) !}^{\frac{1+2s}{s}}+C_\eps\tilde C_0^{2(m-4)}     \com{ (m-4) !}^{\frac{1+2s}{s}}\\
  \leq& 	 \eps  \int_0^1  t^{2\kappa\inner{  m-2}}\norm{  (a^{1/2})^w   f_{m,\delta}}_{L^2}^2 dt + C_\eps\tilde C_0^{2(m-4)}     \com{ (m-4) !}^{\frac{1+2s}{s}},
  \end{aligned}
   \end{equation*}
   the second inequality following from the similar argument for proving \eqref{keyes} and  \eqref{lae},  the third inequality using \eqref{onspa} and the last inequality holding because of $\tilde C_0>C_0^2$ by assumption. 
  Now we derive the upper bound for $\mathcal J_{2,2}.$   Observe  
  \begin{eqnarray*}
  	(a^w+\mathcal R)f =-\Gamma (\mu^{1/2}, f)-\Gamma (f,  \mu^{1/2})=-\mathcal T(\mu^{1/2}, f, \mu^{1/2})-\mathcal T (f,  \mu^{1/2}, \mu^{1/2}),
  \end{eqnarray*}
  recalling the trilinear operator $\mathcal T$ is defined by \eqref{matht}.
 Thus using again Leibniz formula gives
  \begin{multline*}
  	\big[a^w+\mathcal R, \ \partial_{v_1}^m\big]f=\sum_{1\leq j\leq m}\ \sum_{0\leq k\leq j} {m\choose j}{j\choose k} \mathcal T(\partial_{v_1}^{j-k} \mu^{1/2}, \ \partial_{v_1}^{m-j} f,  \ \partial_{v_1}^{k} \mu^{1/2})\\
  	+\sum_{1\leq j\leq m}\ \sum_{0\leq k\leq j} {m\choose j}{j\choose k} \mathcal T(\partial_{v_1}^{m-j} f, \ \partial_{v_1}^{j-k} \mu^{1/2}, \ \partial_{v_1}^{k} \mu^{1/2}).
  \end{multline*}
This, along with \eqref{anfordis},  enables to use similar upper bounds for the trilinear operator $\mathcal T$  as in \eqref{tripleest} and \eqref{tripleest+}  to compute 
  \begin{multline*}
 \mathcal  J_{2,2}  
   \leq  C \inner{ \int_0^1  t^{2\kappa\inner{  m-2}}\norm{(a^{1/2})^w  f_{m,\delta}}_{L^2}^2\ dt}^{1/2}   \\
   \times   \sum_{1\leq j\leq m}\frac{m!}{j!(m-j)!} \inner{\tilde L^{j+1}j!}\Big( \int_0^1  t^{2\kappa\inner{  m-2}}  \norm{ ( a^{1/2})^w \partial_{v_1}^{m-j}f}_{L^2}^2 \ dt\Big)^{1/2},
  \end{multline*}
  with $\tilde  L$ a constant depending only on the constant  $L$ given in \eqref{anfordis}. 	  
Moreover  as for the last   factor in the above inequality, we use the assumption \eqref{asonv} to compute
 \begin{equation*}
 \begin{aligned}
 	& \sum_{1\leq j\leq m}\frac{m!}{j!(m-j)!}  \inner{\tilde L^{j+1}j!}\Big( \int_0^1  t^{2\kappa\inner{  m-2}}  \norm{ ( a^{1/2})^w \partial_{v_1}^{m-j}f}_{L^2}^2 \ dt\Big)^{1/2}\\
 	 \leq & C \bigg\{ \sum_{1\leq j<[m/2]}+ \sum_{[m/2]\leq j\leq m-4}\bigg\} \frac{m!}{(m-j)!}  \tilde L^{j+1} \inner{\int_0^1  t^{2\kappa\inner{  m-j-2}}\norm{ (a^{1/2})^w \partial_{v_1}^{m-j}f}_{L^2}^2\ dt}^{1\over2}\\ 
 	&+C \sum_{m-3\leq j\leq m}\frac{m!}{(m-j)!}  \tilde L^{j+1}\inner{ \int_0^1 t^{2\kappa} \norm{ (a^{1/2})^w \partial_{v_1}^{m-j}f}_{L^2}^2\ dt}^{1/2}\\
 	 \leq  & C  \bigg\{ \sum_{1\leq j<[m/2]}+ \sum_{[m/2]\leq j\leq m-4}\bigg\}\frac{m!}{ (m-j)!} \tilde L^{j+1} \tilde C_0^{m-j-3} \big((m-j-4)!\big)^{\frac{1+2s}{2s}} +C\big(3\tilde L\big)^{m+1} \\
 	 \leq  & C   \tilde C_0^{m-4} \big((m-4)!\big)^{\frac{1+2s}{2s}}  \bigg\{ \sum_{1\leq j<[m/2]}+ \sum_{[m/2]\leq j\leq m-4}\bigg\}\frac{m^4}{(m-j)^4}  \Big( \frac{ \tilde L}{\tilde C_0}\Big)^{j-1}  \tilde  L^2 +C\big(3\tilde L\big)^{m+1}\\
 	 \leq  & C   \tilde C_0^{m-4} \big((m-4)!\big)^{\frac{1+2s}{2s}}   \tilde  L^2 \Big(\sum_{1\leq j<[m/2]} \Big( \frac{ \tilde L}{\tilde C_0}\Big)^{j-1}   +  \sum_{ [m/2]\leq j\leq m-4}   \Big( \frac{ 4\tilde L}{\tilde C_0}\Big)^{j-1} \Big)+C\big(3\tilde L\big)^{m+1}\\
 	 \leq  & C \tilde C_0^{m-4} \big((m-4)!\big)^{\frac{1+2s}{2s}},   
 	\end{aligned}
 	\end{equation*}
 	where the constant $C$ in the last line depends on the constant $L$ given in  \eqref{anfordis}  and  the last inequality holds since we can choose  $\tilde C_0\geq \max\set{8\tilde L, (3\tilde L)^6}.$   Combining these inequalities we conclude
 	\begin{eqnarray*}
 		\mathcal J_{2,2}\leq \eps  \int_0^1  t^{2\kappa\inner{  m-2}}\norm{(a^{1/2})^w  f_{m,\delta}}_{L^2}^2\ dt +C_\eps   \tilde C_0^{2(m-4)} \big((m-4)!\big)^{\frac{1+2s}{s}}.
 	\end{eqnarray*}
This along with the upper bound for $\mathcal J_{2,1}$ gives, for any $\eps>0,$
 \begin{eqnarray*}
 	\mathcal J_{2}\leq\eps  \int_0^1  t^{2\kappa\inner{  m-2}}\norm{(a^{1/2})^w  f_{m,\delta}}_{L^2}^2\ dt +C_\eps   \tilde C_0^{2(m-4)} \big((m-4)!\big)^{\frac{1+2s}{s}}.
 \end{eqnarray*}
 
 \medskip
  \noindent\underline{\it Estimate on    $\mathcal J_3$}.  It is can be treated in a similar way as $\mathcal J_2$ but the argument is direct and much more simpler.     
  In fact  observe
  \begin{multline*}
  	\big[ v\cdot \partial_x,   \     \big(1+\delta\abs{D_{v}}^2\big)^{-1} \big]\partial_{v_1}^m\\
  	= -2i\sum_{1\leq j\leq 3}\underbrace{ \big(1+\delta\abs{D_{v}}^2\big)^{-1}\delta \partial_{v_j}\partial_{v_1}}_{\textrm{uniformly bounded  w.r.t. }\delta}\big(1+\delta\abs{D_{v}}^2\big)^{-1}\partial_{x_j}\partial_{v_1}^{m-1}.
  	  \end{multline*}
  	   This enables to use the argument for treating $\mathcal J_{2,1}$ with slight modification, to get  
   \begin{eqnarray*}
 		&&\int_0^1  t^{2\kappa\inner{  m-2}}\Big|\Big(\big[  v\cdot \partial_x,   \     \big(1+\delta\abs{D_{v}}^2\big)^{-1} \big]\partial_{v_1}^m f,\ ~   f_{m,\delta}\Big)_{L^2}\Big|\ dt\\
 		&\leq&  \eps  \int_0^1  t^{2\kappa\inner{  m-2}}\norm{  (a^{1/2})^w   f_{m,\delta}}_{L^2}^2 dt + C_\eps\tilde C_0^{2(m-4)}     \com{ (m-4) !}^{\frac{1+2s}{s}}.
 \end{eqnarray*}
Moreover observe 
  	  \begin{eqnarray*}
  	&&\big[ a^w+\mathcal R,   \     \big(1+\delta\abs{D_{v}}^2\big)^{-1} \big]\partial_{v_1}^mf\\
  	&=&- \big(1+\delta\abs{D_{v}}^2\big)^{-1}\big[ a^w+\mathcal R,   \   1-\delta\Delta_v \big]  \big(1+\delta\abs{D_{v}}^2\big)^{-1} \partial_{v_1}^mf\\
  	&=&  \big(1+\delta\abs{D_{v}}^2\big)^{-1}\big[ a^w+\mathcal R,   \ \delta\Delta_v      \big]  f_{m,\delta}.
  	  \end{eqnarray*}
Then following the argument for treating $\mathcal J_{2,2}$ above and observing $ \big(1+\delta\abs{D_{v}}^2\big)^{-1} \delta \Delta_v$ is uniformly bounded in $L^2$ w.r.t. $\delta,$ we have 
  	  \begin{eqnarray*}
 		&& \int_0^1  t^{2\kappa\inner{  m-2}}\Big|\Big(\big[  a^w+\mathcal R,   \     \big(1+\delta\abs{D_{v}}^2\big)^{-1} \big]\partial_{v_1}^m f,\ ~   f_{m,\delta}\Big)_{L^2}\Big|\ dt\\
 		&\leq & \eps  \int_0^1  t^{2\kappa\inner{  m-2}}\norm{  (a^{1/2})^w   f_{m,\delta}}_{L^2}^2 dt+C_\eps  \tilde C_0^{2(m-4)} \big((m-4)!\big)^{\frac{1+2s}{s}}.
 \end{eqnarray*}
  	As a result combining the above estimates gives, for any $\eps>0,$  
   \begin{eqnarray*}
 	 \mathcal J_{3}\leq \eps  \int_0^1  t^{2\kappa\inner{  m-2}}\norm{(a^{1/2})^w  f_{m,\delta}}_{L^2}^2\ dt  +C_\eps   \tilde C_0^{2(m-4)} \big((m-4)!\big)^{\frac{1+2s}{s}}.
 \end{eqnarray*}
 It then follows from the upper bounds for    $\mathcal J_1$-$\mathcal J_3$  that
  \begin{eqnarray*}
&&\int_0^1  t^{2\kappa\inner{  m-2}} \abs{\inner{P  f_{m,\delta}, \   f_{m,\delta}}_{L^2}} dt\\
&\leq&  \inner{\eps  +\epsilon_0 C }	 \int_0^1 t^{2 \kappa  (m-2) }   \norm{(a^{1/2})^w    f _{m,\delta}}_{L^2}^2 dt
	 +   C_\eps   \tilde C_0^{2(m-4)}  \com{(m-4)!}^{\frac{1+2s}{s}}.
 \end{eqnarray*}
 This along with \eqref{tildereterm} and \eqref{lae} yields that  for any $0<t\leq 1$ we have,  supposing $\epsilon_0$ is small enough and choosing $\eps$ small as well, 
 \begin{eqnarray*}
 	 t^{2\kappa\inner{  m-2}} \norm{  f_{m,\delta}(t)}_{L^2}^2+ 
   \int_0^1  t^{2\kappa\inner{  m-2}} \norm{  (a^{1/2})^w  f_{m,\delta}}_{L^2}^2 \ dt \leq    C\tilde C_0^{2(m-4)}  \com{(m-4)!}^{\frac{1+2s}{s}}.
 \end{eqnarray*}
 Thus letting $\delta\rightarrow 0$ we obtain
  \begin{eqnarray*}
 	&&  \sup_{0<t\leq 1}t^{\kappa\inner{  m-2}} \norm{  \partial_{v_1}^mf(t)}_{L^2}+ 
 \inner{  \int_0^1  t^{2\kappa\inner{  m-2}} \norm{  (a^{1/2})^w   \partial_{v_1}^mf(t)}_{L^2}^2\ dt}^{1/2}\\
 & \leq&     C\tilde C_0^{m-4}  \com{(m-4)!}^{\frac{1+2s}{2s}}.
 \end{eqnarray*}
 The remaining argument is just the same as that in the proof of Proposition \ref{pro3.2} and we conclude that for any $\alpha$ with $\abs\alpha=m,$
 \begin{eqnarray*}
 	&&  \sup_{t>0}\phi(t)^{\kappa\inner{  m-2}} \norm{  \partial_{v}^\alpha f(t)}_{L^2}+ 
 \inner{  \int_0^{+\infty} \phi( t)^{2\kappa\inner{  m-2}} \norm{  (a^{1/2})^w   \partial_{v}^\alpha f(t)}_{L^2}^2\ dt}^{1/2}\\
 & \leq&     \tilde C_0^{m-3}  \com{(m-4)!}^{\frac{1+2s}{2s}}.
 \end{eqnarray*}
  Thus the proof of Proposition \ref{progevofv} is completed.  
    \end{proof}

\begin{proof}
	[Proof of Theorem \ref{maith}] By virtue of Proposition  \ref{progevofv} we  can  follow the the same procedure for proving Theorem \ref{thm3.1},   to obtain 
	 	\begin{equation}\label{fiestei}
	\forall\ \abs{\alpha}\geq 0,\quad	\sup_{t>0} \phi(t)^{\frac{1+2s}{2s}\abs\alpha}\norm{ \partial_{v}^{\alpha} f(t)}_{L^2} \leq \tilde C_{0}^{\abs\alpha+1}\inner{{\abs\alpha!}}^{\frac{1+2s}{2s}}
	\end{equation}
	for some $\tilde C_0>0.$ Let $C_0>0$ be the constant given in Theorem \ref{thm3.1} and   choose  
	\begin{eqnarray*}
		C=\max\set{2^{\frac{1+2s}{2s}}C_0, \ 2^{\frac{1+2s}{2s}}\tilde C_0}.
	\end{eqnarray*}
	Then for any $\alpha,\beta\in\mathbb Z_+^3$ we use Theorem \ref{thm3.1} and \eqref{fiestei} as well as the fact that $(m+n)!\leq 2^{m+n}m!n!$ for any positive integers $m$ and $n,$  to compute
	\begin{eqnarray*}
	&& \phi(t)^{\frac{1+2s}{2s}(\abs\alpha+\abs\beta)}\norm{ \partial_{x}^{\alpha} \partial_v^\beta f(t)}_{L^2}\\
	&\leq& \phi(t)^{\frac{1+2s}{2s}(\abs\alpha+\abs\beta)}\norm{ \partial_{x}^{2\alpha} f(t)}_{L^2}^{1/2}\norm{   \partial_v^{2\beta }f(t)}_{L^2}^{1/2}\\
		&\leq& \Big(\phi(t)^{\frac{1+2s}{2s}2\abs\alpha}\norm{ \partial_{x}^{2\alpha} f(t)}_{L^2}\Big)^{1/2}\Big(\phi(t)^{\frac{1+2s}{2s}2\abs\beta}\norm{ \partial_{x}^{2\beta} f(t)}_{L^2}\Big)^{1/2}\\
		&\leq&\inner{  C_{0}^{2\abs\alpha+1}\com{{(2\abs\alpha)!}}^{\frac{1+2s}{2s}}}^{1/2}\inner{\tilde C_{0}^{2\abs\beta+1}\com{(2\abs\beta)!}^{\frac{1+2s}{2s}}}^{1/2}\\
		&\leq&  C_{0}^{\abs\alpha+1/2} 2^{\frac{1+2s}{2s}\abs\alpha} \inner{ \abs\alpha !}^{\frac{1+2s}{2s}} \tilde C_{0}^{\abs\beta+1/2} 2^{\frac{1+2s}{2s}\abs\beta} \inner{ \abs\beta !}^{\frac{1+2s}{2s}}\\
		&\leq & C^{\abs\alpha+\abs\beta+1}  \com{ (\abs\alpha+\abs\beta) !}^{\frac{1+2s}{2s}}.
	\end{eqnarray*}
	The proof is thus completed. 
\end{proof}

\appendix
\section{Some facts on Symbolic calculus}\label{secapp}

\subsection{Weyl-H\"ormander calculus}

We recall here  some  notations and basic facts  of  symbolic
calculus, and refer to
\cite[Chapter 18]{Hormander85} or \cite{MR2599384}  for detailed discussions on the pseudo-differential
calculus. 

From now on  let $M$ be an admissible weight function w.r.t. the flat metric $\abs{dv}^2+\abs{d\eta}^2$, that is the weight function $M$ satisfies the following conditions:
\begin{enumerate}[align=right, leftmargin=*,  label=(\alph*)] 
\item (slowly varying condition) there exists a constant $ \delta$ such that
\begin{eqnarray*}
\forall \ X, Y\in \mathbb R^6_{v,\eta}, ~\abs{X-Y}\leq \delta \Longrightarrow  M(X) \approx M(Y);
\end{eqnarray*}
\item (temperance)  there exist two constants $C$ and $N$ such that
\begin{eqnarray*}
\forall~ X, Y\in \mathbb R_{v,\eta}^6, \quad M(X)/M(Y) \leq C
\comii{X-Y}^N.
\end{eqnarray*}
\end{enumerate}
Considering  symbols $q(\xi, v,\eta)$ as a function of $(v,\eta)$ with
parameters $\xi$,    we say that  $q\in S\inner{M, \abs{dv}^2+\abs{d\eta}^2}$ uniformly
with respect to $\xi$, if
\[
    \forall~ \alpha, \beta\in\mathbb Z_+^3,~~\forall~v,\eta\in\mathbb R^3,\quad
\abs{\partial_v^\alpha\partial_\eta^\beta q(\xi,v,\eta)}\leq C_{\alpha,
\beta} M,
\]
with $ C_{\alpha,\beta}$ a constant depending only on $\alpha$ and $\beta$, but
independent of $\xi$.
For simplicity of notations, in the following discussion, we
omit the parameters dependence in the  symbols, and by $q\in S(M, \abs{dv}^2+\abs{d\eta}^2)$ we always mean
that $q$  satisfies the above inequality, uniformly with respect to
$\xi$. The space $S(M,\abs{dv}^2+\abs{d\eta}^2)$ endowed with the semi-norms
\begin{eqnarray*}\label{seminorm}
 \norm{q}_{k; S(M, \abs{dv}^2+\abs{d\eta}^2)}= \max_{0 \leq \abs\alpha+\abs\beta \leq k} \sup_{(v,\eta)\in \mathbb
 R^6} \abs{M(v,\eta)^{-1}\partial_v^\alpha\partial_\eta^\beta q (v,\eta)},
\end{eqnarray*}
becomes a Fr\'echet space.
 Let $q\in \mathcal S'(\mathbb R_v^3\times\mathbb R_\eta^3)$ be a tempered distribution and let $t\in\mathbb R$.   the operator ${\rm op}_t q$ is an operator from  $ \mathcal S(\mathbb R_v^3)$ to $ \mathcal S'(\mathbb R_v^3),$  whose Schwartz kernel $K_t$ is defined by the oscillatory integral:    
 \[
    K_t  (z, z')= (2\pi)^{-3} \int_{\mathbb R^3} e^{i (z-z') \cdot\zeta}q((1-t)z+tz', \zeta) d\zeta.
  \]
In particular we denote $q(v, D_v)={\rm op}_0q$ and $q^w={\rm op}_{1/2}q$.
 Here $q^w$ is called  the Weyl quantization of symbol $q$.

 An elementary property to be used
frequently is  the $L^2$ continuity  theorem in the class
$S\inner{1,~\abs{dv}^2+\abs{d\eta}^2}$, see \cite[Theorem 2.5.1]{MR2599384} for instance,  which says that  there exists a constant $C$
and  a positive integer $N$ depending only the dimension,
such that
\begin{equation}\label{bdness}
    \forall ~u\in L^2,\quad \norm{ q^w u}_{L^2}\leq C \norm{q}_{N;
      S(1, \abs{dv}^2+\abs{d\eta}^2)}\norm{u}_{L^2}.
\end{equation}
Let us   recall   the composition formula of Weyl
quantization.  Given $p_i\in S(M_i,\ \abs{dv}^2+\abs{d\eta}^2)$ we have
\begin{eqnarray*}
  p_1^wp_2^w=(p_1\sharp p_2)^w
\end{eqnarray*}
with $p_1\sharp p_2 \in S\inner{M_1M_2,\ \abs{dv}^2+\abs{d\eta}^2}$  admitting the expansion
\begin{multline*}
   p_1\sharp p_2=p_1 p_2\\
   + \int_0^1\iint e^{-2 i \sigma(Y-Y_1, Y-Y_2)/\theta}
   \frac{1}{2i}\sigma (\partial_{Y_1}, \partial_{Y_2}) p_1(Y_1)
   p_2 (Y_2) dY_1 dY_2 d\theta/(\pi\theta)^6,
\end{multline*}
where $\sigma$ is the symplectic form in $\mathbb R^6$ given by
\[
   \sigma \inner{(z,\zeta), (\tilde z,\tilde \zeta)}=\zeta\cdot
   \tilde z-\tilde \zeta \cdot z.
\]
  Finally we mention that $q^w$ is self-adjoint in $L^2$ if $q$ is real-valued symbol.

\subsection{Wick quantization}\label{subwick}
Finally let us  recall some  basic properties of the Wick quantization, which is  also called anti-Wick in \cite{MR883081}.   The importance in studying the Wick quantization lies in the facts that  positive symbols give rise to positive operators.   There are several equivalent ways of defining Wick quantization and one is defined in terms of coherent states.   The coherent states method  essentially reduces the partial differential operators to  ODEs,  by virtue of  the Wick calculus.     

 Let $Y = (v, \eta)$  be a point in $
\mathbb R^{6}$.  The Wick quantization of a symbol $q$  is  given by
\begin{eqnarray*}
	q^{{\rm Wick}} =(2\pi)^{-3}\int_{\mathbb R^6} q(Y) \Pi_{Y}\ dY,
\end{eqnarray*}
where  $\Pi_Y$ is the projector associated to the Gaussian $\varphi_Y$ which is defined by 
\[
   \varphi_Y(z)=\pi^{-3/4}e^{-\frac{1}{2}\abs{z-v}^2}e^{i z \cdot\eta /2},\quad z\in
\mathbb R^3.
\]
 The main property of the Wick quantization is its positivity, i.e.,
\begin{eqnarray*}
  q(v,\eta)\geq 0 ~~\,\textrm{for all}~ (v,\eta)\in\mathbb R^{6} ~{\rm implies}
~~\,q^{\rm Wick}\geq 0.
\end{eqnarray*}
According to Theorem 24.1 in \cite{MR883081}, the Wick and Weyl
quantizations of a symbol $q$ are linked by
the following identities
\begin{eqnarray*}
  q^{\rm Wick}=\inner{q* \pi^{-3} e^{-
\abs{\cdot}^2 }}^w=q^w+r^w
\end{eqnarray*}
with
\[
  r(Y)= \pi^{-3} \int_0^1 \int_{\mathbb R^{6}} (1-\theta)q''(Y+\theta Z )Z^2e^{- 
\abs{Z}^2}d Zd\theta.
\]
As a result,  $q^{\rm Wick} $ is a bounded operator in $L^2$ if $q\in S(1, g)$ due to \eqref{bdness} and self-adjoint in $L^2$ if $q$ is real-valued symbol.

We also recall the following composition
formula obtained in the proof of    \cite[Proposition 3.4]{MR1957713},  
\begin{eqnarray} \label{11082406}
  q_1^{\rm Wick}q_2^{\rm Wick}=\Big(q_1q_2- q_1'\cdot q_2'+
  \frac{1}{i}\big \{q_1,~q_2\big\}\Big)^{\rm Wick},
\end{eqnarray}
  provided one of  $q_1 $ and $q_2$ is a polynomial of order $1$.  The notation $\set{q_1,q_2}$ denotes the Poisson
bracket defined by
\begin{eqnarray} \label{11051505}
  \set{q_1,~q_2}=\frac{\partial q_1}{\partial\eta}\cdot\frac{\partial q_2}{\partial
v}-\frac{\partial q_1}{\partial v}\cdot\frac{\partial q_2}{\partial \eta}.
\end{eqnarray}


\begin{thebibliography}{10}

\bibitem{MR2556715}
R.~Alexandre.
\newblock A review of {B}oltzmann equation with singular kernels.
\newblock {\em Kinet. Relat. Models}, 2(4):551--646, 2009.

\bibitem{ADVW00}
R.~Alexandre, L.~Desvillettes, C.~Villani, and B.~Wennberg.
\newblock Entropy dissipation and long-range interactions.
\newblock {\em Arch. Ration. Mech. Anal.}, 152(4):327--355, 2000.

\bibitem{AHL}
R.~{Alexandre}, F.~{H{\'e}rau}, and W.-X. {Li}.
\newblock {Global hypoelliptic and symbolic estimates for the linearized
  Boltzmann operator without angular cutoff}.
\newblock {\em Preprint, arXiv:1212.4632}, Dec. 2012.

\bibitem{MR2679369}
R.~Alexandre, Y.~Morimoto, S.~Ukai, X.~C.-J., and T.~Yang.
\newblock Regularizing effect and local existence for the non-cutoff
  {B}oltzmann equation.
\newblock {\em Arch. Ration. Mech. Anal.}, 198(1):39--123, 2010.

\bibitem{ALEXANDRE20082013}
R.~Alexandre, Y.~Morimoto, S.~Ukai, C.-J. Xu, and T.~Yang.
\newblock Uncertainty principle and kinetic equations.
\newblock {\em Journal of Functional Analysis}, 255(8):2013 -- 2066, 2008.

\bibitem{MR2793203}
R.~Alexandre, Y.~Morimoto, S.~Ukai, C.-J. Xu, and T.~Yang.
\newblock The {B}oltzmann equation without angular cutoff in the whole space:
  {II}, {G}lobal existence for hard potential.
\newblock {\em Anal. Appl. (Singap.)}, 9(2):113--134, 2011.

\bibitem{MR2847536}
R.~Alexandre, Y.~Morimoto, S.~Ukai, C.-J. Xu, and T.~Yang.
\newblock The {B}oltzmann equation without angular cutoff in the whole space:
  qualitative properties of solutions.
\newblock {\em Arch. Ration. Mech. Anal.}, 202(2):599--661, 2011.

\bibitem{AMUXY1}
R.~Alexandre, Y.~Morimoto, S.~Ukai, C.-J. Xu, and T.~Yang.
\newblock The boltzmann equation without angular cutoff in the whole space: I,
  global existence for soft potential.
\newblock {\em Journal of Functional Analysis}, 262(3):915 -- 1010, 2012.

\bibitem{AlexMorUkaiXuYang}
R.~Alexandre, Y.~Morimoto, S.~Ukai, C.-J. Xu, and T.~Yang.
\newblock Local existence with mild regularity for the {B}oltzmann equation.
\newblock {\em Kinet. Relat. Models}, 6(4):1011--1041, 2013.

\bibitem{MR2149928}
R.~Alexandre and M.~Safadi.
\newblock Littlewood-{P}aley theory and regularity issues in {B}oltzmann
  homogeneous equations. {I}. {N}on-cutoff case and {M}axwellian molecules.
\newblock {\em Math. Models Methods Appl. Sci.}, 15(6):907--920, 2005.

\bibitem{MR2476677}
R.~Alexandre and M.~Safadi.
\newblock Littlewood-{P}aley theory and regularity issues in {B}oltzmann
  homogeneous equations. {II}. {N}on cutoff case and non {M}axwellian
  molecules.
\newblock {\em Discrete Contin. Dyn. Syst.}, 24(1):1--11, 2009.

\bibitem{MR1857879}
R.~Alexandre and C.~Villani.
\newblock On the {B}oltzmann equation for long-range interactions.
\newblock {\em Comm. Pure Appl. Math.}, 55(1):30--70, 2002.

\bibitem{Barbaroux2017}
J.-M. Barbaroux, D.~Hundertmark, T.~Ried, and S.~Vugalter.
\newblock Gevrey smoothing for weak solutions of the fully nonlinear
  homogeneous boltzmann and kac equations without cutoff for maxwellian
  molecules.
\newblock {\em Archive for Rational Mechanics and Analysis}, 225(2):601--661,
  2017.

\bibitem{Bouchut02}
F.~Bouchut.
\newblock Hypoelliptic regularity in kinetic equations.
\newblock {\em J. Math. Pure Appl.}, 81:1135--1159, 2002.

\bibitem{MR1069558}
C.~Cercignani.
\newblock {\em Mathematical methods in kinetic theory}.
\newblock Plenum Press, New York, second edition, 1990.

\bibitem{MR1307620}
C.~Cercignani, R.~Illner, and M.~Pulvirenti.
\newblock {\em The mathematical theory of dilute gases}, volume 106 of {\em
  Applied Mathematical Sciences}.
\newblock Springer-Verlag, New York, 1994.

\bibitem{MR2425602}
H.~Chen, W.-X. Li, and X.~C.-J.
\newblock Propagation of {G}evrey regularity for solutions of {L}andau
  equations.
\newblock {\em Kinet. Relat. Models}, 1(3):355--368, 2008.

\bibitem{MR2557895}
H.~Chen, W.-X. Li, and X.~C.-J.
\newblock Analytic smoothness effect of solutions for spatially homogeneous
  {L}andau equation.
\newblock {\em J. Differential Equations}, 248(1):77--94, 2010.

\bibitem{MR2763329}
H.~Chen, W.-X. Li, and C.-J. Xu.
\newblock Gevrey hypoellipticity for a class of kinetic equations.
\newblock {\em Comm. Partial Differential Equations}, 36(4):693--728, 2011.

\bibitem{ChenDesvHe07}
Y.~Chen, L.~Desvillettes, and L.~He.
\newblock Smoothing effects for classical solutions of the full landau
  equation.
\newblock {\em Archive for Rational Mechanics and Analysis}, 193(1):21--55,
  2009.

\bibitem{MR1324404}
L.~Desvillettes.
\newblock About the regularizing properties of the non-cut-off {K}ac equation.
\newblock {\em Comm. Math. Phys.}, 168(2):417--440, 1995.

\bibitem{Desv97}
L.~Desvillettes.
\newblock Regularization properties of the {$2$}-dimensional non-radially
  symmetric non-cutoff spatially homogeneous {B}oltzmann equation for
  {M}axwellian molecules.
\newblock {\em Transport Theory Statist. Phys.}, 26(3):341--357, 1997.

\bibitem{MR2465814}
L.~Desvillettes, G.~Furioli, and E.~Terraneo.
\newblock Propagation of {G}evrey regularity for solutions of the {B}oltzmann
  equation for {M}axwellian molecules.
\newblock {\em Trans. Amer. Math. Soc.}, 361(4):1731--1747, 2009.

\bibitem{DesvVillani00-2}
L.~Desvillettes and C.~Villani.
\newblock On the spatially homogeneous {L}andau equation for hard potentials.
  {I}. {E}xistence, uniqueness and smoothness.
\newblock {\em Comm. Partial Differential Equations}, 25(1-2):179--259, 2000.

\bibitem{DesvWennberg04}
L.~Desvillettes and B.~Wennberg.
\newblock Smoothness of the solution of the spatially homogeneous {B}oltzmann
  equation without cutoff.
\newblock {\em Comm. Partial Differential Equations}, 29(1-2):133--155, 2004.

\bibitem{DipernaLions89}
R.~J. DiPerna and P.-L. Lions.
\newblock On the {C}auchy problem for {B}oltzmann equations: global existence
  and weak stability.
\newblock {\em Ann. of Math. (2)}, 130(2):321--366, 1989.

\bibitem{MR3485915}
L.~Glangetas, H.-G. Li, and C.-J. Xu.
\newblock Sharp regularity properties for the non-cutoff spatially homogeneous
  {B}oltzmann equation.
\newblock {\em Kinet. Relat. Models}, 9(2):299--371, 2016.

\bibitem{GIMV}
F.~{Golse}, C.~{Imbert}, C.~{Mouhot}, and A.~{Vasseur}.
\newblock {Harnack inequality for kinetic Fokker-Planck equations with rough
  coefficients and application to the Landau equation}.
\newblock {\em Accepted by Annali della Scuola Normale Superiore di Pisa,
  Classe di Scienze, arXiv:1607.08068}.

\bibitem{MR923047}
F.~Golse, P.-L. Lions, B.~t. Perthame, and R.~Sentis.
\newblock Regularity of the moments of the solution of a transport equation.
\newblock {\em J. Funct. Anal.}, 76(1):110--125, 1988.

\bibitem{MR808622}
F.~Golse, B.~t. Perthame, and R.~Sentis.
\newblock Un r\'esultat de compacit\'e pour les \'equations de transport et
  application au calcul de la limite de la valeur propre principale d'un
  op\'erateur de transport.
\newblock {\em C. R. Acad. Sci. Paris S\'er. I Math.}, 301(7):341--344, 1985.

\bibitem{MR2784329}
P.~T. Gressman and R.~M. Strain.
\newblock Global classical solutions of the {B}oltzmann equation without
  angular cut-off.
\newblock {\em J. Amer. Math. Soc.}, 24(3):771--847, 2011.

\bibitem{MR2807092}
P.~T. Gressman and R.~M. Strain.
\newblock Sharp anisotropic estimates for the {B}oltzmann collision operator
  and its entropy production.
\newblock {\em Adv. Math.}, 227(6):2349--2384, 2011.

\bibitem{2017arXiv170705710H}
C.~{Henderson} and S.~{Snelson}.
\newblock { $C^\infty$ smoothing for weak solutions of the inhomogeneous Landau
  equation}.
\newblock {\em Preprint, arXiv:1707.05710}, July 2017.

\bibitem{MR3102561}
F.~H\'erau and W.-X. Li.
\newblock Global hypoelliptic estimates for {L}andau-type operators with
  external potential.
\newblock {\em Kyoto J. Math.}, 53(3):533--565, 2013.

\bibitem{HK2011}
F.~H{\'e}rau and K.~Pravda-Starov.
\newblock Anisotropic hypoelliptic estimates for {L}andau-type operators.
\newblock {\em J. Math. Pures et Appl.}, 95:513--552, 2011.

\bibitem{Hormander85}
L.~H{\"o}rmander.
\newblock {\em The analysis of linear partial differential operators. {III}},
  volume 275 of {\em Grundlehren der Mathematischen Wissenschaften}.
\newblock Springer-Verlag, Berlin, 1985.

\bibitem{MR2425608}
Z.~Huo, Y.~Morimoto, S.~Ukai, and T.~Yang.
\newblock Regularity of solutions for spatially homogeneous {B}oltzmann
  equation without angular cutoff.
\newblock {\em Kinet. Relat. Models}, 1(3):453--489, 2008.

\bibitem{IM2}
C.~{Imbert} and C.~{Mouhot}.
\newblock {H\"older continuity of solutions to hypoelliptic equations with
  bounded measurable coefficients}.
\newblock {\em Preprint, arXiv:1505.04608}, May 2015.

\bibitem{IM}
C.~{Imbert} and C.~{Mouhot}.
\newblock {A toy nonlinear model in kinetic theory}.
\newblock {\em Preprint, arXiv:1801.07891}, Jan. 2018.

\bibitem{2018arXiv180406135I}
C.~{Imbert}, C.~{Mouhot}, and L.~{Silvestre}.
\newblock {Decay estimates for large velocities in the Boltzmann equation
  without cut-off}.
\newblock {\em Preprints, arXiv:1804.06135}, Apr. 2018.

\bibitem{IS}
C.~{Imbert} and L.~{Silvestre}.
\newblock {The weak Harnack inequality for the Boltzmann equation without
  cut-off}.
\newblock {\em Accepted by Journal of the European Mathematical Society,
  arXiv:1608.07571}, Aug. 2016.

\bibitem{MR1957713}
N.~Lerner.
\newblock The {W}ick calculus of pseudo-differential operators and some of its
  applications.
\newblock {\em Cubo Mat. Educ.}, 5(1):213--236, 2003.

\bibitem{MR2599384}
N.~Lerner.
\newblock {\em Metrics on the phase space and non-selfadjoint
  pseudo-differential operators}.
\newblock Birkh\"auser Verlag, Basel, 2010.

\bibitem{MR2876831}
N.~Lerner, Y.~Morimoto, and K.~Pravda-Starov.
\newblock Hypoelliptic estimates for a linear model of the {B}oltzmann equation
  without angular cutoff.
\newblock {\em Comm. Partial Differential Equations}, 37(2):234--284, 2012.

\bibitem{LERNER2015459}
N.~Lerner, Y.~Morimoto, K.~Pravda-Starov, and C.-J. Xu.
\newblock Gelfand-shilov and {G}evrey smoothing effect for the spatially
  inhomogeneous non-cutoff {K}ac equation.
\newblock {\em Journal of Functional Analysis}, 269(2):459 -- 535, 2015.

\bibitem{MR3680949}
H.-G. Li and C.-J. Xu.
\newblock The {C}auchy problem for the radially symmetric homogeneous
  {B}oltzmann equation with {S}hubin class initial datum and {G}elfand-{S}hilov
  smoothing effect.
\newblock {\em J. Differential Equations}, 263(8):5120--5150, 2017.

\bibitem{MR3193940}
W.-X. Li.
\newblock Global hypoelliptic estimates for fractional order kinetic equation.
\newblock {\em Math. Nachr.}, 287(5-6):610--637, 2014.

\bibitem{Lions98}
P.-L. Lions.
\newblock R\'{e}gularit\'{e} et compacit\'{e} pour des noyaux de collision de
  {B}oltzmann sans troncature angulaire.
\newblock {\em C. R. Acad. Sci. Paris S\'{e}r. I Math.}, 326:37--41, 1998.

\bibitem{MR2679746}
Y.~Morimoto and S.~Ukai.
\newblock Gevrey smoothing effect of solutions for spatially homogeneous
  nonlinear {B}oltzmann equation without angular cutoff.
\newblock {\em J. Pseudo-Differ. Oper. Appl.}, 1(1):139--159, 2010.

\bibitem{MR2476686}
Y.~Morimoto, S.~Ukai, C.-J. Xu, and T.~Yang.
\newblock Regularity of solutions to the spatially homogeneous {B}oltzmann
  equation without angular cutoff.
\newblock {\em Discrete Contin. Dyn. Syst.}, 24(1):187--212, 2009.

\bibitem{MR2359105}
Y.~Morimoto and C.-J. Xu.
\newblock Hypoellipticity for a class of kinetic equations.
\newblock {\em J. Math. Kyoto Univ.}, 47(1):129--152, 2007.

\bibitem{MR2523694}
Y.~Morimoto and C.-J. Xu.
\newblock Ultra-analytic effect of {C}auchy problem for a class of kinetic
  equations.
\newblock {\em J. Differential Equations}, 247(2):596--617, 2009.

\bibitem{MR3325244}
Y.~Morimoto and T.~Yang.
\newblock Smoothing effect of the homogeneous {B}oltzmann equation with measure
  valued initial datum.
\newblock {\em Ann. Inst. H. Poincar\'e Anal. Non Lin\'eaire}, 32(2):429--442,
  2015.

\bibitem{Mou2}
C.~Mouhot.
\newblock Explicit coercivity estimates for the {B}oltzmann and {L}andau
  operators.
\newblock {\em Comm. Partial Differential Equations}, 31:1321--1348, 2006.

\bibitem{Mou1}
C.~Mouhot and R.~Strain.
\newblock Spectral gap and coercivity estimates for linearized {B}oltzmann
  collision operators without angular cutoff.
\newblock {\em J. Math. Pures Appl.}, 87:515--535, 2007.

\bibitem{MR883081}
M.~A. Shubin.
\newblock {\em Pseudodifferential operators and spectral theory}.
\newblock Springer Series in Soviet Mathematics. Springer-Verlag, Berlin, 1987.
\newblock Translated from the Russian by Stig I. Andersson.

\bibitem{2018arXiv180510264}
S.~{Snelson}.
\newblock {The inhomogeneous Landau equation with hard potentials}.
\newblock {\em Preprint, arXiv:1805.10264}, May 2018.

\bibitem{Ukai84}
S.~Ukai.
\newblock Local solutions in gevrey classes to the nonlinear boltzmann equation
  without cutoff.
\newblock {\em Japan J. Appl. Math.}, 1(1):141--156, 1984.

\bibitem{Villani1999}
C.~Villani.
\newblock Regularity estimates via the entropy dissipation for the spatially
  homogeneous boltzmann equation without cut-off.
\newblock {\em Rev. Mat. Iberoam.}, 15(2):335--352, 1999.

\bibitem{Villani02}
C.~Villani.
\newblock A review of mathematical topics in collisional kinetic theory.
\newblock In {\em Handbook of mathematical fluid dynamics, Vol. I}, pages
  71--305. North-Holland, Amsterdam, 2002.

\end{thebibliography}

\end{document}